\pdfoutput=1
\documentclass[11pt]{article}
\usepackage[margin=1in]{geometry}
\newcommand{\citep}[0]{\cite}
\usepackage{graphicx,verbatim,color}

\usepackage{amsmath}
\usepackage{amssymb}
\usepackage{multirow}
\usepackage{mathrsfs}
\usepackage{amsfonts}
\usepackage{xcolor}
\usepackage{mathtools}
\usepackage[normalem]{ulem}

\usepackage{graphicx}
\usepackage{epsfig}
\usepackage{comment}
\usepackage{epstopdf}
\usepackage{cite}
\usepackage{amsthm}\usepackage{dsfont}\usepackage{array}\usepackage{mathrsfs}\usepackage{comment}
\usepackage[hidelinks]{hyperref}
\usepackage{breakurl}
\usepackage{amsmath}
\usepackage{amssymb}
 \usepackage{multirow}
 \usepackage{mathrsfs}
\usepackage{amsfonts}
\newcommand{\trace}{{\mathbf{Tr}}}

\newcommand{\ep}{\epsilon}

\newcommand{\diag}{{\rm diag}}

\DeclareFontFamily{U}{mathx}{\hyphenchar\font45}
\DeclareFontShape{U}{mathx}{m}{n}{
      <5> <6> <7> <8> <9> <10>
      <10.95> <12> <14.4> <17.28> <20.74> <24.88>
      mathx10
      }{}
\DeclareSymbolFont{mathx}{U}{mathx}{m}{n}
\DeclareMathAccent{\widecheck}{\mathalpha}{mathx}{"71}

\newcommand{\eps}{\epsilon}

\newcommand\bbR{\ensuremath{\mathbb{R}}} 
\newcommand\bbE{\ensuremath{\mathbb{E}}} 

\DeclarePairedDelimiter\abs{\lvert}{\rvert}%
\DeclarePairedDelimiterX{\norm}[1]{\lVert}{\rVert}{#1}

\DeclarePairedDelimiterX{\infdivx}[2]{(}{)}{%
  #1\;\delimsize\|\;#2%
}

\newcommand{\trK}{{\mathbf{K}}}
\newcommand{\trG}{{\mathbf{G}}}
\newcommand{\trDelta}{{\mathbf{\Delta}}}

\newcommand{\trX}{{\mathbf{X}}}
\newcommand{\trY}{{\mathbf{Y}}}

\newcommand\calH{\ensuremath{\mathcal{H}}} 

\newcounter{relctr}[subsection] 

\newcommand\labelrel[2]{%
  \begingroup
    \refstepcounter{relctr}%
    \stackrel{\textnormal{(\roman{relctr})}}{\mathstrut{#1}}%
    \originallabel{#2}%
  \endgroup
}
\AtBeginDocument{\let\originallabel\label}
\usepackage[utf8]{inputenc} 
\usepackage[T1]{fontenc}    
\usepackage{hyperref}       
\usepackage{url}            
\usepackage{booktabs}       
\usepackage{amsfonts}       
\usepackage{nicefrac}       

\usepackage{enumitem}
\usepackage{microtype}      

\usepackage{float}
\usepackage{caption}
\usepackage{subcaption}

\usepackage{algorithm}

\usepackage{algorithmic}

\usepackage[nameinlink]{cleveref}
\crefname{equation}{}{}
\crefname{theorem}{Theorem}{Theorems}
\crefname{corollary}{Corollary}{Corollaries}
\crefname{example}{Example}{Examples}
\crefname{assumption}{Assumption}{Assumptions}
\crefname{lemma}{Lemma}{Lemmas}
\crefname{proposition}{Proposition}{Propositions}
\crefname{figure}{Figure}{Figures}
\crefname{table}{Table}{Tables}
\crefname{section}{Section}{Sections}
\crefname{appendix}{Appendix}{Appendices}
\Crefname{equation}{}{}
\Crefname{theorem}{Theorem}{Theorems}
\Crefname{corollary}{Corollary}{Corollaries}
\Crefname{example}{Example}{Examples}
\Crefname{lemma}{Lemma}{Lemma}
\Crefname{proposition}{Proposition}{Proposition}
\Crefname{figure}{Figure}{Figures}
\Crefname{table}{Table}{Tables}
\Crefname{section}{Section}{Sections}
\Crefname{appendix}{Appendix}{Appendices}

\newcommand\numberthis{\stepcounter{equation}\tag{\theequation}}
\let\oldnl\nl
\newcommand{\nonl}{\renewcommand{\nl}{\let\nl\oldnl}}

\usepackage{authblk}
\title{On Controller Reduction in Linear Quadratic Gaussian Control with Performance Bounds\thanks{The work of Zhaolin Ren and Na Li is supported by NSF CNS 2003111, NSF AI institute 2112085, and ONR YIP N00014-19-1-2217. The work of Yang Zheng is supported by NSF ECCS-2154650. The work of Maryam Fazel was supported by NSF TRIPODS II 2023166, CCF 1839291, CCF 2007036, and CCF 2212261. Emails: zhaolinren@g.harvard.edu (Zhaolin Ren); zhengy@eng.ucsd.edu (Yang Zheng); mfazel@uw.edu (Maryam Fazel); and nali@seas.harvard.edu (Na Li).}}
\author[1]{Zhaolin Ren}
\author[2]{Yang Zheng}
\author[3]{Maryam Fazel}
\author[1]{Na Li}
\affil[1]{School of Engineering and Applied Science, Harvard University, USA.}
\affil[2]{ECE Department, University of California San Diego, La Jolla, USA}
\affil[3]{ECE Department, University of Washington, Seattle, USA.}

\begin{document}

\maketitle

\theoremstyle{plain}
\newtheorem{lemma}{\textbf{Lemma}}
\newtheorem{theorem}{\textbf{Theorem}}
\newtheorem{proposition}{\textbf{Proposition}}
\newtheorem{corollary}{\textbf{Corollary}}
\newtheorem{assumption}{\textbf{Assumption}}
\newtheorem{example}{\textbf{Example}}
\newtheorem{definition}{\textbf{Definition}}
\newtheorem{conjecture}{\textbf{Conjecture}}
\newtheorem{claim}{\textbf{Claim}}
\newtheorem{remark}{\textbf{Remark}}
\newtheorem{question}{\textbf{Question}}
\theoremstyle{definition}
\newcommand{\tr}{{{\mathsf T}}}
\newcommand{\mK}{{\mathsf{K}}}

\maketitle

\begin{abstract}%
 The problem of controller reduction has a rich history in control theory. Yet, many questions remain open. 
In particular, there exist very few results on the order reduction of general non-observer based controllers and the subsequent quantification of the closed-loop~performance. 
Recent developments in model-free policy optimization for Linear Quadratic Gaussian (LQG) control have highlighted the importance of this question. 
In this paper, we first propose a new set of sufficient conditions ensuring that a perturbed controller remains internally stabilizing. Based on this result, we illustrate how to perform order reduction of general non-observer based controllers using balanced truncation and modal truncation. 
We also provide explicit bounds on the LQG performance of the reduced-order controller. Furthermore, for single-input-single-output (SISO) systems, we introduce a new controller reduction technique by truncating \textit{unstable} modes. 
We illustrate our theoretical results with numerical simulations.  Our results will serve as valuable tools to design direct policy search algorithms for control problems with partial observations. 
\end{abstract}

\section{Introduction}



In many control applications, 
low-order controllers are~often preferred over high-order controllers, because they are simpler to maintain, more interpretable, and computationally less demanding \citep{anderson1987controller}. Thus, given a high-order controller, one often would like to approximate it using a lower-order controller that still stabilizes the plant whilst performing similarly on relevant closed-loop performance metrics, such as the Linear Quadratic Gaussian (LQG) cost. This problem is known as \emph{controller reduction}. Traditional approaches to controller reduction in LQG control have typically centered on reducing the order of observer-based controllers 
and providing error bounds between the performance of the truncated controller and that of the original controller \citep{anderson1987controller, zhou1995performance}. 

However, the problem of order-reduction for general non-observer based controllers has been less studied, especially in the context of LQG control. Recent progress in model-free policy optimization for linear control has highlighted the importance of order-reduction for general controllers~\citep{zheng2022escaping}. 
In particular, a natural problem in model-free policy optimization is to learn an optimal policy iteratively using policy gradient methods \citep{hu2022towards}. It has recently been shown that the optimization landscape of LQG control may contain saddle points in state-space dynamic controllers \citep{tang2021analysis}. While vanilla policy gradient ensures the convergence to stationary points under mild assumptions, these stationary points may be saddle points that are sub-optimal. 
As shown very recently in \cite{zheng2022escaping}, when a saddle point corresponds to a non-minimal controller, 
it is possible to escape the saddle point by finding a lower-order controller and 
adding a suitable random perturbation during policy gradient. 
%
It is thus natural to consider order-reduction~for general non-observer based controllers, such that we find a lower-order controller with approximately equivalent or lower LQG cost. 
Moreover, policy gradient for LQG control may also meet unstable controllers\footnote{A dynamic controller that has unstable modes itself but internally stabilizes the plant.}, but the results on order-reduction for unstable controllers are far less complete \citep{anderson1987controller,liu1986controller}. This motivates the main questions in this paper:
\vspace{-1mm}
\begin{center}
\begin{enumerate} \setlength{\itemsep}{0pt} 
    \item \textit{Can we perform controller reduction on general, possibly unstable, LQG controllers, such that the reduced-order controller remains internally stabilizing?}
    \item \textit{Can we provide explicit error bounds on the LQG performance of the reduced-order controller compared to the original controller?}
\end{enumerate}
\end{center}
\vspace{-1mm}
These questions are not only relevant for the reasons relating to policy optimization of LQG control \citep{tang2021analysis}, but are interesting in their own right for the model and controller reduction literature~\citep{anderson1987controller,obinata2012model}. 

\textbf{Our contributions.} In this paper, we provide positive answers to both questions. We first identify a novel set of sufficient conditions that ensure the stability of a perturbed controller (\cref{theorem:truncated_K_still_stable_conditions}), and then derive a new bound on the LQG cost of a perturbed controller under the assumption that the truncated component is stable and appropriately small (\cref{theorem:stable_K_truncation_result}). For general multiple-input and multiple-output (MIMO) systems, building on \cref{theorem:truncated_K_still_stable_conditions} and \cref{theorem:stable_K_truncation_result}, we then show (in \cref{section:controller_reduction_via_balanced_truncation} and \cref{section:controller_reduction_modal_truncation_stable} respectively) how balanced truncation and modal truncation may be applied to general (non observer-based, possibly unstable) LQG controllers to yield lower-order controllers with bounded LQG performance gap (compared to that of the original controller). Furthermore, for single-input and single-output (SISO) systems, we discuss in \cref{section:order_reduction_for_unstable_controller_SISO} how internal stability may be preserved even when the reduced-order controller has fewer unstable poles than the original controller. This opens the path of controller reduction via truncating unstable poles, a novel controller reduction technique which we illustrate both theoretically and empirically. Finally, in \cref{section:near-pole-zero cancellation_to_jordan block}, we characterize the connection between the existence of a ``small'' Jordan block with near pole-zero cancellation in SISO systems. 
This allows us to connect modal reduction with order reduction for minimal systems that are close to being non-minimal.




\subsection{Related work} 

Our work follows a long line of work in the controller reduction literature. 
A classical result is that when a controller $\trK$ and its reduced-order counterpart $\trK_r$ have the same number of unstable poles and no poles on the imaginary axis, assuming some other conditions involving the truncated component $\trDelta = \trK - \trK_r$ hold (see \cref{lemma:K_perturbed_still_stable_conditions_old}), the reduced-order controller $\trK_r$ will stabilize the original system \citep{anderson1987controller,doyle1981multivariable}. Thus, one common approach to controller reduction is to truncate the stable part of a controller, whilst keeping its unstable part intact. Popular methods to perform truncation of the stable part of a controller include modal truncation \citep{skelton1988dynamic,jonckheere1982singular}, balanced truncation \citep{enns1984model}, and Hankel norm approximation \citep{glover1983robust}. When the difference of the stable portion and its truncated portion satisfies a (frequency-weighted) error bound (see \cref{lemma:K_perturbed_still_stable_conditions_old}; cf. \cite[Section II.A]{anderson1987controller}), it guarantees that the truncated controller 
remains internally stabilizing. However, there appear to be no existing results providing an error bound on the LQG cost of the truncated controller from such a procedure for general (possibly non observer-based) controllers. In contrast, for observer-based controllers, there has been a significant line of work based on coprime factorizations \citep{vidyasagar1975coprime}, which not only yields reduced-order controllers that are internally stabilizing, but also guarantees LQG performance bounds for the resulting truncated controllers (cf. \cite{liu1986controller,anderson1987controller}). However, a key limitation is that these method only work for observer-based controllers.

A closely related but distinct research direction to controller reduction is the topic of open-loop model reduction \citep{obinata2012model,benner2017model,antoulas2000survey,antoulas2005approximation}. While techniques in model reduction often overlap with those in controller reduction (e.g. balanced truncation \citep{pernebo1982model}, based on the theory of balanced realization in \cite{moore1981principal}), modal truncation, and Hankel norm approximation \cite{zhou1995frequency}), tight approximation bounds for the open-loop controller do not translate, a priori, to tight performance bounds for the closed-loop system. For instance, while there has been work studying model reduction for general unstable systems \citep{zhou1999balanced, yang1993model, mirnateghi2013model}, it is unclear if such techniques can produce a stabilizing lower-order approximation of an unstable (but stabilizing) controller. 

Another important related topic is policy optimization for linear control problems. There has been significant recent work studying policy optimization for linear quadratic (LQ) control problems, for linear-quadratic-regulator (LQR) \citep{fazel2018global}, $\mathcal{H}_2$ linear control with $\mathcal{H}_\infty$ guarantees \citep{zhang2020policy,Hu2023connectivity},  as well as LQG problems \citep{tang2021analysis,zheng2022escaping}. In particular, as we explained earlier, the considerations outlined in \cite{zheng2022escaping} on escaping saddle points of the LQG problem was an important motivation for our work, where controller reduction is required. See \cite{hu2022towards} for a recent review.  



\subsection{Paper outline}  \label{subsection:notation}


The rest of this paper is structured as follows.  
We present the problem statement in \Cref{section:problem-statement}. In \Cref{section:robust-stability}, we first introduce \cref{theorem:truncated_K_still_stable_conditions}, which provides sufficient conditions such that a perturbed controller $\trK_r$ of $\trK$ is internally stabilizing. We further derive an upper bound on the LQG cost $J(\trK_r)$ (\cref{theorem:stable_K_truncation_result}). In  \cref{section:controller_reduction_via_balanced_truncation}, we study balanced truncation on the stable part of a controller. In \cref{section:order_reduction_modal}, we discuss modal truncation, where \cref{section:controller_reduction_modal_truncation_stable} studies modal truncation on the stable part of a controller, and \cref{section:order_reduction_for_unstable_controller_SISO} discusses controller reduction via truncation of unstable poles for SISO systems. 
\Cref{section:near-pole-zero cancellation_to_jordan block} presents the connection between near pole-zero cancellation and small Jordan block.  
We provide  numerical experiments to illustrate our theoretical results in \Cref{section:examples}. We conclude the paper in \Cref{section:conclusion}.  

\subsection{Notation} 
We denote the set of real-rational proper stable transfer functions as $\mathcal{RH}_{\infty}$ (i.e., all the poles are on the open left-half complex plane). For simplicity, we omit the dimension of transfer matrices. 
The state-space realization of a transfer function $\mathbf{G}(s) = C(sI - A)^{-1}B + D$ is denoted as
 $$ \mathbf{G}(s) = \left[\begin{array}{c|c} A & B \\\hline
     C & D\end{array}\right].
     $$ 
     We define the $\mathcal{L}_\infty$ norm for a transfer function $\trG(s)$ as 
$
\norm*{\trG}_{\mathcal{L}_\infty} \coloneqq \sup_{w \in \bbR} \sigma_{\max}(\trG(jw)),
$
where $\sigma_{\max}(\cdot)$ denotes the maximum singular value. 
When $\trG \in \mathcal{RH}_\infty$, its $\mathcal{H}_\infty$ norm is the same as its $\mathcal{L}_\infty$ norm \cite[Chapter 4.3]{zhou1996robust}. 
We define the $\mathcal{L}_2$ norm for $\trG(s)$ as 
$$\norm*{\trG}_{\mathcal{L}_2} \coloneqq \sqrt{\frac{1}{2\pi} \int_{-\infty}^\infty \trace(\trG(-jw)^\tr \trG(jw) dw}.$$
When $\trG$ is stable and 
strictly proper, its $\mathcal{H}_2$ norm is the same as its $\mathcal{L}_2$ norm \cite[Chapter 4.3]{zhou1996robust}.  We denote the set of transfer functions in $\mathcal{R}\mathcal{H}_\infty$ with finite $\mathcal{L}_2$ norm as $\mathcal{R}\mathcal{H}_2$. 
Throughout our proofs later, we often utilize submultiplicative-like inequalities to bound the $\mathcal{H}_\infty$ or $\mathcal{H}_2$ norm of products of transfer functions. For completeness, we summarize a list of such inequalities in \cref{appendix:background_norm_bounds}.

\section{Preliminaries and Problem Statement} \label{section:problem-statement}


\subsection{General non-observer based controllers}


Consider a strictly proper linear time-invariant (LTI) plant\footnote{For simplicity, we assume that there is a common $B$ in front of both the $u(t)$ and $w(t)$ terms.}
\begin{equation}\label{eq:LTI}
    \begin{aligned}
        \dot{x}(t) &= A x(t) + B u(t) + Bw(t),\\
        y(t)   &= Cx + v(t),
    \end{aligned}
\end{equation}
where $x(t) \in \mathbb{R}^{n},u(t)\in \mathbb{R}^{m},y(t)\in \mathbb{R}^{p}$ are the state vector, control action, and measurement vector at time $t$, respectively; $w(t)\in \mathbb{R}^{m}$ and $v(t)\in \mathbb{R}^{p}$ are external disturbances on the state and measurement vectors at time $t$, respectively. One basic yet fundamental control problem is to design a feedback controller (or policy) to stabilize the plant \cref{eq:LTI}. A standard approach for this problem is to use an observer-based controller of the form
\begin{equation}\label{eq:observer}
\begin{aligned}
\dot{\xi}(t) &= A \xi(t) +  B u(t) + L(y(t) - C\xi(t)) \\
u(t) &= -K\xi(t),
\end{aligned}
\end{equation}
where $\xi(t) \in \mathbb{R}^n$ is an estimated state, $L \in \mathbb{R}^{n \times p}$ is an observer gain, and $K \in \mathbb{R}^{m \times n}$ is a feedback gain. The observer and feedback gains are chosen such that $A - LC$ and $A - BK$ are stable, and this guarantees the closed-loop internal stability when applying the controller \cref{eq:observer} to the plant \cref{eq:LTI} \cite[Chapter 3.5]{zhou1996robust}. Note that the order of this observer-based controller must be the same with the system plant (i.e., the controller state $\xi(t)$ and the system state $x(t)$ have the same dimension). Order reduction for controllers in the form of \eqref{eq:observer} is discussed in \cite{liu1986controller,anderson1987controller,zhou1995performance}. 

In this paper, we consider a general non-observer based dynamic controller of the form
\begin{equation}\label{eq:Dynamic_Controller}
    \begin{aligned}
        \dot \xi(t) &= A_{\mK}\xi(t) + B_{\mK}y(t), \\
        u(t) &= C_{\mK}\xi(t),
    \end{aligned}
\end{equation}
where $\xi(t) \in \mathbb{R}^q$ is the internal state of the controller, and $A_{\mK},B_{\mK},C_{\mK}$ are matrices of proper dimensions that specify the dynamics of the controller. The dimension $q$ of the internal control variable $\xi$ is called the order of the dynamical controller~\cref{eq:Dynamic_Controller}. The controller in \cref{eq:Dynamic_Controller} is more suitable for model-free policy optimization as it does not explicitly depend on the system dynamics \citep{tang2021analysis,zheng2022escaping}.  
It is clear that the observer-based controller \cref{eq:observer} is a special case of \cref{eq:Dynamic_Controller} by taking $q = n$, and $A_\mK = A - BK - LC, B_{\mK} = L, C_{\mK} = -K$. 
By combining~\cref{eq:Dynamic_Controller} with~\cref{eq:LTI}, the closed-loop system is internally stable if and only if the closed-loop matrix 
\begin{equation} \label{eq:internal-stability}
    A_{\mathrm{cl}}:=\begin{bmatrix}
    A & BC_{\mK} \\
    B_{\mK}C & A_{\mK}
    \end{bmatrix}
\end{equation}
is stable \cite[Lemma 5.2]{zhou1996robust}.

\subsection{Problem statement}
Given an internally stabilizing controller $(A_{\mK}, B_\mK, C_\mK)$ satisfying \eqref{eq:internal-stability}, the \textit{controller reduction problem} \citep{anderson1987controller} is to find a new controller $(\hat{A}_{\mK}, \hat{B}_\mK, \hat{C}_\mK)$ of lower order $\hat{q} < q$ such that it still internally stabilizes the plant and does not significantly affect the closed-loop performance.  
In particular, we consider a normalized LQG control performance~\citep{mustafa1990controller,jonckheere1983new}, defined as 
\begin{equation} \label{eq:LQG-cost}
    J =\lim_{T \to \infty} \mathbb{E} \left[\frac{1}{T}\int_{0}^T x(t)^\tr C^\tr C {x}(t) + u(t)^\tr u(t) dt \right],
\end{equation}

We make the following two assumptions.
\begin{assumption} \label{assumption-minimal}
The plant \eqref{eq:LTI} is minimal, i.e., $(A,B)$ is controllable and $(C,A)$ is observable.
\end{assumption}

\begin{assumption} \label{assumption-white}
In plant~\eqref{eq:LTI}, the signals $w(t)\in \mathbb{R}^{m}$ and $v(t)\in \mathbb{R}^{p}$ are zero mean Gaussian white noise, each with a spectrum equal to the identity. 
\end{assumption}

\Cref{assumption-minimal} is standard and guarantees the existence of internally stabilizing controllers\footnote{The existence of internally stabilizing controllers only requires stabilizability and detectablity.}. If the plant is not minimal, we can always perform a lower-order minimal realization before designing a controller. 
\Cref{assumption-white} was used  to define the  \textit{normalized LQG control problem} in~\citep{mustafa1990controller,jonckheere1983new}. As we shall see next, this assumption simplifies the expression of LQG cost \cref{eq:LQG-cost} in the frequency domain. For the controller reduction problem, it may not be easy to work directly with the internal stability condition \cref{eq:internal-stability} in the state-space domain due to non-uniqueness of state-space realizations.  
It is more convenient to consider equivalent conditions in the frequency domain. In particular, the controller \cref{eq:Dynamic_Controller} can be represented as a transfer function $\mathbf{K} := C_{\mK}(sI - A_\mK)^{-1}B_\mK$. Let us define a new performance signal $\tilde{y} = Cx$. Some simple manipulations show that the closed-loop transfer function from $(w,v)$ to $(\tilde{y},u)$ is  
\begin{align}
\label{eq:y_tilde_u_w_v_freq_rship}
    \begin{bmatrix}
    \mathbf{\tilde{y}} \\
    \mathbf{u}
    \end{bmatrix} = \begin{bmatrix}
    (I - \mathbf{G}\mathbf{K})^{-1} \mathbf{G} & (I - \mathbf{G}\mathbf{K})^{-1} \mathbf{G} \mathbf{K} \\
    \mathbf{K}(I - \mathbf{G}\mathbf{K})^{-1} \mathbf{G} &  \mathbf{K} (I - \mathbf{G}\mathbf{K})^{-1}
    \end{bmatrix} \begin{bmatrix}
    \mathbf{w} \\
    \mathbf{v}
    \end{bmatrix},
\end{align}
\normalsize
where $\mathbf{G}(s) = C(sI - A)^{-1} B$. Then, we have the following condition for internal stability.
\begin{lemma}[\!{\cite[Lemma 5.3]{zhou1996robust}}] \label{lemma:stability-frequency}
The controller $\mathbf{K}$ in \eqref{eq:Dynamic_Controller} internally stabilizes the plant~\eqref{eq:LTI} if and only if the closed-loop transfer function from $(w,v)$ to $(\tilde{y},u)$ is stable\footnote{The standard result in \cite[Lemma 5.3]{zhou1996robust} uses a slightly different set of closed-loop transfer functions. Simple manipulations via $(I - \mathbf{G}\mathbf{K})^{-1}  = I + (I - \mathbf{G}\mathbf{K})^{-1} \mathbf{G} \mathbf{K}$ can show the equivalence.}, i.e.,
$$ \begin{bmatrix}
    (I - \mathbf{G}\mathbf{K})^{-1} \mathbf{G} & (I - \mathbf{G}\mathbf{K})^{-1} \mathbf{G} \mathbf{K} \\
    \mathbf{K}(I - \mathbf{G}\mathbf{K})^{-1} \mathbf{G} &  \mathbf{K} (I - \mathbf{G}\mathbf{K})^{-1}
    \end{bmatrix} \in \mathcal{RH}_\infty
$$
\end{lemma}
Furthermore, it is not difficult to see that the normalized LQG control cost \eqref{eq:LQG-cost} can be expressed conveniently in the frequency domain. For completeness, we provide a proof of \cref{lemma:LQG-cost} in \cref{appendix:proof_of_LQG-cost}.
\begin{lemma} \label{lemma:LQG-cost}
Under \Cref{assumption-white}, given an internally stabilizing controller $\mathbf{K}$ in \eqref{eq:Dynamic_Controller}, the normalized LQG control cost~\cref{eq:LQG-cost} can be expressed as follows
$$
J(\mathbf{K}) = \left\|\begin{bmatrix}
    (I - \mathbf{G}\mathbf{K})^{-1} \mathbf{G} & (I - \mathbf{G}\mathbf{K})^{-1} \mathbf{G} \mathbf{K} \\
    \mathbf{K}(I - \mathbf{G}\mathbf{K})^{-1} \mathbf{G} &  \mathbf{K} (I - \mathbf{G}\mathbf{K})^{-1}
    \end{bmatrix}\right\|_{\mathcal{H}_2}^2. 
$$
\normalsize
\end{lemma}

\section{Robust stability and LQG performance} \label{section:robust-stability}


Our main goal is to properly perturb the controller $\mathbf{K}$ to get a lower-order controller $\mathbf{K}_r$ such that  the closed-loop performance remains similar. In this section, we present two technical results that underpin our controller reduction results in \Cref{section:controller_reduction_via_balanced_truncation,section:order_reduction_modal}: 1) a new robust stability result (\Cref{theorem:truncated_K_still_stable_conditions}), and 2) an upper bound on the LQG performance for the new controller $\mathbf{K}_r$ (\Cref{theorem:stable_K_truncation_result}).

\subsection{A novel sufficient condition for internal stability} \label{section:a_novel_sufficient_condition}

Classical results on controller reduction often focus on the case when the truncated controller $\mathbf{K}_r$ has the same number of unstable poles as the original controller $\mathbf{K}$ (cf. \cite{liu1986controller,anderson1987controller}). Under such a condition, sufficient conditions are available using the $\mathcal{L}_\infty$ norm of terms involving the difference $\trK_r - \trK$ such that $\trK_r$ can still stabilize $\trG$ if $\trK$ stabilizes $\trG$. In particular, a widely-used condition is as follows. 

\begin{lemma}[{\cite[Section II.A]{anderson1987controller}}]
\label{lemma:K_perturbed_still_stable_conditions_old}
Let $\mathbf{G}$ be the transfer function of an LTI plant \cref{eq:LTI}, the controller $\mathbf{K}$ \cref{eq:Dynamic_Controller} internally stabilize the plant, and $\mathbf{K}_r$ be another dynamic controller. Denote $\trDelta \coloneqq \trK_r - \trK$. If
\begin{enumerate}
    \item $\mathbf{K}$ and $\mathbf{K}_r$ have the same number of poles in $Re(s) > 0$, and no poles on the imaginary axis, 
    \item either {\small $\norm*{\trDelta \mathbf{G}(I - \mathbf{K}\mathbf{G})^{-1}}_{\mathcal{L}_\infty}\! < \!1$}  or  {\small $\norm*{(I - \mathbf{G}\mathbf{K})^{-1}\mathbf{G} \trDelta}_{\mathcal{L}_\infty\!}\! < \!1$}, 
 \end{enumerate}
 \normalsize
 then $\mathbf{K}_r$ also internally stabilizes the plant $\mathbf{G}$.
\end{lemma}
This classical result underpins many controller reduction techniques in the literature; see \cite{anderson1987controller} for a review. \Cref{lemma:K_perturbed_still_stable_conditions_old} can be proved by combining Nyquist stability criterion \cite[Section IV]{doyle1981multivariable} with a classical result \cite[Theorem 5.7]{zhou1996robust}. 
For completeness, we provide a proof in \Cref{appendix:lemma_3_proof}. 
If the controller $\mathbf{K}$ is stable in the first place, then the first condition in \Cref{lemma:K_perturbed_still_stable_conditions_old} can be naturally satisfied by using any stable controller $\mathbf{K}_r$. When the controller $\mathbf{K}$ is unstable (i.e., $A_\mK$ in \cref{eq:Dynamic_Controller} is unstable), we might always need to preserve the unstable part in $\mathbf{K}$ in order to use \Cref{lemma:K_perturbed_still_stable_conditions_old}. However, it is unclear if it is necessary for a lower-order truncated controller $\trK_r$ to have the same number of unstable poles as $\trK$ in order to maintain closed-loop stability. 

In this section, we provide a novel set of sufficient conditions in \Cref{theorem:truncated_K_still_stable_conditions} ensuring that $\trK_r$ is still stabilizing, which makes no explicit assumptions on whether $\trK_r$ and $\trK$ have the same number of unstable poles. This technical result may be of independent interest. 


\vspace{1mm}
\begin{theorem}
\label{theorem:truncated_K_still_stable_conditions}
Let $\mathbf{G}$ be the transfer function of an LTI plant \cref{eq:LTI}, the controller $\mathbf{K}$ \cref{eq:Dynamic_Controller} internally stabilize the plant, and let $\mathbf{K}_r$ denote another controller. Denote $\trDelta \coloneqq \trK_r - \trK$. 
If $\trDelta (I - \trG \trK)^{-1}$ is stable and
\begin{align}
\label{eq:truncated_K_still_stable_conditions}
    { \max\left\{ \norm*{(I - \trG \trK)^{-1} \trG \trDelta}_{\calH_{\infty}} , \norm*{\trDelta (I - \trG \trK)^{-1} \trG}_{\calH_{\infty}} \right\} < 1,} 
\end{align}
\normalsize
then, $\trK_r$ also internally stabilizes $\trG$. 
\end{theorem}

Unlike the proof of \Cref{lemma:K_perturbed_still_stable_conditions_old} that is based on Nyquist stability \cite[Section IV]{doyle1981multivariable},  \Cref{theorem:truncated_K_still_stable_conditions} can be proved directly from \Cref{lemma:stability-frequency}. If we can show 
\begin{equation} \label{eq:H-G-Kr}
 \begin{bmatrix}
    (I - \mathbf{G}\mathbf{K}_r)^{-1} \mathbf{G} & (I - \mathbf{G}\mathbf{K}_r)^{-1} \mathbf{G} \mathbf{K}_r \\
    \mathbf{K}_r(I - \mathbf{G}\mathbf{K}_r)^{-1} \mathbf{G} &  \mathbf{K}_r (I - \mathbf{G}\mathbf{K}_r)^{-1}
    \end{bmatrix} \in \mathcal{RH}_\infty,
\end{equation}
\normalsize
\Cref{lemma:stability-frequency} confirms that $\trK_r$ internally stabilizes $\trG$. Since $\mathbf{K}$ internally stabilizes $\mathbf{G}$, we know that 
\begin{equation} \label{eq:H-G-K}
    \begin{bmatrix}
    (I - \mathbf{G}\mathbf{K})^{-1} \mathbf{G} & (I - \mathbf{G}\mathbf{K})^{-1} \mathbf{G} \mathbf{K} \\
    \mathbf{K}(I - \mathbf{G}\mathbf{K})^{-1} \mathbf{G} &  \mathbf{K} (I - \mathbf{G}\mathbf{K})^{-1}
    \end{bmatrix} \in \mathcal{RH}_\infty.
\end{equation}
\normalsize
Motivated by \cite[Appendix C]{zheng2021sample}, the key idea in our proof is to relate the transfer functions in \eqref{eq:H-G-Kr} with those in \cref{eq:H-G-K}, and then to show that each of the four subblocks in \eqref{eq:H-G-Kr} is stable.

\begin{proof}
For notational simplicity, denote 
\begin{equation} \label{eq:definition-X-Y}
\mathbf{Y} \coloneqq (I - \mathbf{G}\mathbf{K})^{-1}, \; \mathbf{X} \coloneqq (I - \mathbf{G}\mathbf{K})^{-1} \mathbf{G}.
\end{equation}
By this definition, we naturally have $\mathbf{X} =  \mathbf{Y} \mathbf{G}$, and $\mathbf{G} = \mathbf{Y}^{-1}\mathbf{X}.$ Since $\mathbf{K}$ internally stabilizes $\mathbf{G}$, we know that $\mathbf{Y}$ and $\mathbf{X}$ are both stable. 
Observe that 
\begin{equation} \label{eq:X-Y-indentity}
\trY = (I - \trG \trK)^{-1} \trG \trK + (I - \trG \trK)^{-1} ( I - \trG \trK) = (I - \trG \trK)^{-1} \trG \trK + I,
\end{equation}
which implies $\mathbf{X}\trK = (I - \trG \trK)^{-1} \trG \trK$ is stable.  
By our definition in \cref{eq:definition-X-Y}, the condition \cref{eq:truncated_K_still_stable_conditions} is equivalent to $\norm*{\trX \trDelta}_{\calH_{\infty}} < 1$, $\norm*{\trDelta \trX}_{\calH_{\infty}} < 1$ and $\trDelta \trY$ is stable. 
According to the small-gain theorem \cite[Theorem 9.1]{zhou1996robust}, we know that 
both $
    (I - \trX \trDelta)^{-1}$ \text{and} $(I -  \trDelta \trX)^{-1}
$
exist and they are stable.  

We now proceed to show stability of the four subblocks in \eqref{eq:H-G-Kr}. 
A helpful calculation is that
\begin{align}
   (I - \mathbf{G}\mathbf{K}_r)^{-1} =  (I - \mathbf{G}(\mathbf{K} + \mathbf{\Delta}))^{-1} \nonumber 
   &= (I - \mathbf{Y}^{-1} \mathbf{X}(\mathbf{K} + \mathbf{\Delta}))^{-1}\nonumber \\
   &= (\mathbf{Y} - \mathbf{X}(\mathbf{K} + \mathbf{\Delta}))^{-1} \mathbf{Y} \nonumber \\
   &= (I - \mathbf{X} \mathbf{\Delta})^{-1} \mathbf{Y},
   \label{eq:I-GK_r_inv_G_expression_second_appearance}
\end{align}
where we have applied that fact that $\mathbf{Y} = \mathbf{X}\mathbf{K}+ I$ from \cref{eq:X-Y-indentity}. 

First, noting that $\mathbf{Y} \coloneqq (I - \trG \trK)^{-1}$ and using \cref{eq:I-GK_r_inv_G_expression_second_appearance}, we have
\begin{align}
(I - \mathbf{G}\mathbf{K}_r)^{-1}\mathbf{G} = (I - \mathbf{X} \mathbf{\Delta})^{-1} (I - \mathbf{G}\mathbf{K})^{-1}\mathbf{G},
\label{eq:I-GK_r_inv_G_expression}
\end{align}
which is stable since $(I - \mathbf{X} \mathbf{\Delta})^{-1}$ and $(I - \mathbf{G}\mathbf{K})^{-1}\mathbf{G}$ are both stable. 

Next observe that 
\begin{align}
       (I - \mathbf{G}\mathbf{K}_r)^{-1}\mathbf{G} \mathbf{K}_r 
       &\labelrel={eq:make_use_of_I-GK_r_inv_G_expression} (I - \mathbf{G}\mathbf{K}_r)^{-1}\mathbf{G} \mathbf{K} + (I - \trX \trDelta)^{-1} \trX \mathbf{\Delta} \nonumber \\
       &\labelrel={eq:make_use_of_I-GK_r_inv_G_expression_2nd} (I - \trX \trDelta)^{-1}(I - \trG \trK)^{-1}\mathbf{G} \mathbf{K} + (I - \trX \trDelta)^{-1} \trX \mathbf{\Delta},\label{eq:I-GK_r_inv_G_Kr_expression} 
    \end{align}
where we have applied  \cref{eq:I-GK_r_inv_G_expression_second_appearance} to derive (\ref{eq:make_use_of_I-GK_r_inv_G_expression}) and (\ref{eq:make_use_of_I-GK_r_inv_G_expression_2nd}). Since $(I - \trX \trDelta)^{-1}$, $(I - \trG \trK)^{-1}\mathbf{G} \mathbf{K}$, and $\trX \mathbf{\Delta}$ are all stable, we know that $(I - \mathbf{G}\mathbf{K}_r)^{-1}\mathbf{G} \mathbf{K}_r$ is stable. 

Next, we consider the term $\mathbf{K}_r (I - \mathbf{G} \mathbf{K}_r)^{-1}$. We have 
    \begin{align*}
         \mathbf{K}_r (I - \mathbf{G} \mathbf{K}_r)^{-1} 
        = & \ \trK \left( I - \mathbf{X} \mathbf{\Delta} \right)^{-1} \trY + \trDelta \left( I - \mathbf{X} \mathbf{\Delta} \right)^{-1} \trY \\
        \labelrel={eq:used_I-Xdelta_inv_identity} & \ \trK (I +  \mathbf{X} \trDelta(I - \mathbf{X}\trDelta)^{-1})\trY + \trDelta \left( I - \mathbf{X} \mathbf{\Delta} \right)^{-1} \trY \\
        = & \ \trK(I - \mathbf{G}\trK)^{-1}  + \trK (\mathbf{X}\mathbf{\Delta})(I - \mathbf{X} \mathbf{\Delta})^{-1}  \trY + \mathbf{\Delta} (I - \mathbf{X} \mathbf{\Delta})^{-1} \trY\\
        \labelrel={eq:used_push_through_identity} & \ \trK(I - \mathbf{G}\trK)^{-1} + \trK \trX  (I - \trDelta \mathbf{X})^{-1} \mathbf{\Delta}  \trY + (I - \trDelta\mathbf{X} )^{-1}\trDelta \trY. \numberthis \label{eq:Kr(I-GKr)_inv_expression}
    \end{align*}
    To derive (\ref{eq:used_I-Xdelta_inv_identity}), we used the fact that $(I - \mathbf{X}\trDelta)^{-1} = (I - \mathbf{X}\trDelta + \mathbf{X}\trDelta)(I - \mathbf{X}\trDelta)^{-1} = I + \mathbf{X}\trDelta(I - \mathbf{X}\trDelta)^{-1}.$
    To derive (\ref{eq:used_push_through_identity}), we used the push through identity $\trDelta(I - \trX \trDelta)^{-1} = (I - \trDelta \trX)^{-1} \trDelta$.  
It is now clear $\mathbf{K}_r (I - \mathbf{G} \mathbf{K}_r)^{-1}$ is stable, since $\trK(I - \mathbf{G}\trK)^{-1}$, $\trK \trX$,  $(I - \trDelta \mathbf{X})^{-1}$, and $ \mathbf{\Delta}  \trY$ are all stable.

Finally, to study the stability of $\trK_r(I - \trG \trK_r)^{-1} \trG$, we reuse the calculations from  $\trK_r(I - \trG \trK_r)^{-1}$ in \cref{eq:Kr(I-GKr)_inv_expression}, leading to
    \begin{align*}
        \mathbf{K}_r (I - \mathbf{G} \mathbf{K}_r)^{-1} \trG  
        = &\trK(I - \mathbf{G}\trK)^{-1} \trG + (I - \trDelta\mathbf{X})^{-1}\trDelta \trY \trG + \trK \trX  (I - \trDelta \mathbf{X})^{-1} \mathbf{\Delta}  \trY \trG, \\
        = &\trK(I - \mathbf{G}\trK)^{-1} \trG + (I - \trDelta\mathbf{X})^{-1}\trDelta \mathbf{X} + \trK \trX  (I - \trDelta \mathbf{X})^{-1} \mathbf{\Delta}  \mathbf{X},
    \end{align*}
which is stable, since $\trK(I - \mathbf{G}\trK)^{-1} \trG $, $\trK \trX $, $(I - \trDelta \mathbf{X})^{-1}$, and $\mathbf{\Delta}  \mathbf{X}$ are all stable.  
We have thus shown the stability of all four subblocks in \eqref{eq:H-G-Kr}. This completes the proof.
\end{proof}

\begin{remark}[\Cref{lemma:K_perturbed_still_stable_conditions_old} versus \Cref{theorem:truncated_K_still_stable_conditions}]
Both \Cref{lemma:K_perturbed_still_stable_conditions_old} and \Cref{theorem:truncated_K_still_stable_conditions} provide a set of sufficient  conditions for a reduced-order controller to internally stabilize the original plant. We comment here on some similarities and differences between the two sets of sufficient conditions. First, 
condition 2 in \cref{lemma:K_perturbed_still_stable_conditions_old} is stated in terms of the $\mathcal{L}_\infty$ norm, while \cref{eq:truncated_K_still_stable_conditions} in \Cref{theorem:truncated_K_still_stable_conditions} requires the $\mathcal{H}_\infty$ norm, thus condition 2 in \cref{lemma:K_perturbed_still_stable_conditions_old} is in fact a weaker requirement than \cref{eq:truncated_K_still_stable_conditions}\footnote{Note that an unstable transfer function may have finite $\mathcal{L}_\infty$ norm and only stable transfer functions can have finite $\mathcal{H}_\infty$ norm; for the definition of the $\mathcal{L}_\infty$ and $\mathcal{H}_\infty$ norm, the reader can refer to the notation in \Cref{subsection:notation}.}.  In addition, \cref{theorem:truncated_K_still_stable_conditions} requires that $\trDelta(I - \trG \trK)^{-1}$ should be stable, whilst \cref{lemma:K_perturbed_still_stable_conditions_old} makes no such requirement. However, a key advantage of \cref{theorem:truncated_K_still_stable_conditions} over \cref{lemma:K_perturbed_still_stable_conditions_old} is that it does not require $\trK$ and its reduced-order counterpart $\trK_r$ to have the same number of poles in $Re(s) > 0$ or to have no poles on the imaginary axis. This suggests that effective controller reduction may not necessitate preserving all unstable poles of $\trK$ --- this intuition will turn out to be useful for \cref{theorem:modal_truncation_SISO} (\Cref{section:order_reduction_for_unstable_controller_SISO}), which shows that controller reduction via unstable modal truncation is in fact possible.   
\end{remark}

\subsection{A new bound on the perturbed LQG cost} \label{subsection:robustLQGperformance}

\Cref{theorem:truncated_K_still_stable_conditions} presents sufficient conditions to guarantee the closed-loop stability using the reduced-order controller $\trK_r$. In many situations (such as policy optimization for LQG in  \cite{tang2021analysis,zheng2022escaping}), we also need to understand the closed-loop performance under this new controller $\trK_r$. Our next technical result show that if the error $\trDelta \coloneqq \trK_r - \trK$ is stable, the change of the LQG cost \cref{eq:LQG-cost} can also be bounded. The proof builds on the analysis techniques in \cref{theorem:truncated_K_still_stable_conditions}.

\begin{theorem}
\label{theorem:stable_K_truncation_result}
Let $\mathbf{G}$ be the transfer function of an LTI plant \cref{eq:LTI}, the controller $\mathbf{K}$ \cref{eq:Dynamic_Controller} internally stabilize the plant, and $\mathbf{K}_r$ denote another controller. Denote $\trDelta \coloneqq \trK_r - \trK$.  If 
\begin{align}
\label{eq:K_Kr_pole_stability_condition_1}
    \norm*{\trDelta}_{\mathcal{H}_\infty}< \frac{1}{ \norm*{(I - \mathbf{G}\mathbf{K})^{-1} \mathbf{G}}_{\mathcal{H}_\infty}},
\end{align}
\normalsize
 then the controller $\trK_r$ internally stabilizes the plant \cref{eq:LTI}, 
%
    and the resulting LQG cost  \cref{eq:LQG-cost}  satisfies 
     \begin{equation} \label{eq:LQG-cost-bound}
     J(\trK_r) \leq \frac{1}{\left(1 - \norm*{\mathbf{X}}_{\mathcal{H}_{\infty}} \norm*{\mathbf{\Delta}}_{\mathcal{H}_{\infty}}\right)^2} (J(\mathbf{K}) + S_1 + S_2),
     \end{equation}
     \normalsize
where with the notation $\mathbf{Y} := (I - \mathbf{G}\trK)^{-1}$, $\mathbf{X} := (I - \mathbf{G}\trK)^{-1} \mathbf{G}$ and $\mathbf{\Delta} := \mathbf{K}_r - \mathbf{K}$, 
we have
\begin{align}
    &S_1 := 2\norm*{\mathbf{\Delta}}_{\mathcal{H}_\infty} \norm*{\mathbf{X}}_{\mathcal{H}_2}(\norm*{\mathbf{X} \mathbf{K}}_{\mathcal{H}_2}) \label{eq:S1-S2} 
    + 2\norm*{\mathbf{\Delta}}_{\mathcal{H}_2} \!(\!\norm*{\trK \mathbf{Y}}_{\mathcal{H}_2} \norm*{\mathbf{Y}}_{\mathcal{H}_\infty}\! +\! \norm*{\trK \mathbf{X}}_{\mathcal{H}_2} \norm*{\mathbf{X}}_{\mathcal{H}_\infty}\!)\!(1 \!+\! \norm*{\trK \trX}_{\mathcal{H}_{
        \infty}} ),  \\
    &S_2 := \norm*{\mathbf{\Delta}}_{\mathcal{H}_\infty}^2 \norm*{\mathbf{X}}_{\mathcal{H}_2}^2 \!+\! \norm*{\mathbf{\Delta}}_{\mathcal{H}_2}^2 (\norm*{\mathbf{Y}}_{\mathcal{H}_{\infty}}^2 + \norm*{\mathbf{X}}_{\mathcal{H}_{\infty}}^2) (1\! +\! \norm*{\trK \trX}_{\mathcal{H}_{
        \infty}} )^2.  \nonumber
\end{align}
\normalsize
\end{theorem}   

\begin{proof}
We first show via \cref{theorem:truncated_K_still_stable_conditions} that $\trK_r$ internally stabilizes $\trG$. If \cref{eq:K_Kr_pole_stability_condition_1} holds, it follows that 
\begin{align*}
\norm*{(I - \trG \trK)^{-1} \trG \trDelta}_{\mathcal{H}_\infty} &\leq \norm*{(I - \trG \trK)^{-1} \trG }_{\mathcal{H}_\infty}  \norm*{\trDelta}_{\mathcal{H}_\infty} < 1.
\end{align*} 
Similarly, it can be verified that 
$
\norm*{\trDelta(I - \trG \trK)^{-1} \trG}_{\mathcal{H}_\infty} < 1.
$
We also have that $\trDelta(I - \trG \trK)^{-1}$ is stable since $\trDelta$ and $(I - \trG \trK)^{-1}$ are both stable. Therefore, the conditions in \cref{theorem:truncated_K_still_stable_conditions} are all satisfied, which implies then that $\trK_r$ internally stabilizes $\trG$.

We then proceed to upper bound the LQG cost $J(\mathbf{K}_r)$ for the new reduced-order controller $\mathbf{K}_r$. From \Cref{lemma:LQG-cost}, we know that 
\begin{equation*}
    J(\mathbf{K}_r) \!=\! \norm*{(I \!-\! \mathbf{G}\mathbf{K}_r)^{-1}\mathbf{G}}_{\mathcal{H}_2}^2 + \norm*{(I \!-\! \mathbf{G}\mathbf{K}_r)^{-1}\mathbf{G} \trK_r}_{\mathcal{H}_2}^2 + \norm*{\trK_r(I \!-\! \mathbf{G}\mathbf{K}_r)^{-1}\mathbf{G}}_{\mathcal{H}_2}^2 + \norm*{\trK_r(I \!-\! \mathbf{G}\mathbf{K}_r)^{-1}}_{\mathcal{H}_2}^2. 
\end{equation*}
Our strategy is to bound each term above using the norms of the corresponding terms when the original controller $\mathbf{K}$ is applied\footnote{This process will frequently use some standard norm inequalities listed in \Cref{appendix:background_norm_bounds}.}. 

    We first observe that by \cref{eq:I-GK_r_inv_G_expression}, the following inequality  holds
    \begin{align}
        \norm*{(I - \mathbf{G}\mathbf{K}_r)^{-1}\mathbf{G}}_{\mathcal{H}_2} &= \norm*{(I - \mathbf{X} \mathbf{\Delta})^{-1} (I - \mathbf{G}\mathbf{K})^{-1}\mathbf{G}}_{\mathcal{H}_2} \nonumber \\
        &\leq \norm*{(I - \mathbf{X} \mathbf{\Delta})^{-1}}_{\mathcal{H}_{\infty}}\norm*{(I - \mathbf{G}\mathbf{K})^{-1}\mathbf{G}}_{\mathcal{H}_2} \nonumber \\
        &\leq \frac{1}{1 - \norm*{\mathbf{X}}_{\mathcal{H}_{\infty}} \norm*{\mathbf{\Delta}}_{\mathcal{H}_\infty}} \norm*{(I - \mathbf{G}\mathbf{K})^{-1}\mathbf{G}}_{\mathcal{H}_2}, \label{eq:first_K_r_sub-block_H2_diff}
    \end{align}
    where we have applied standard norm inequalities in \Cref{lemma:G1_G_2_norm_ineq_H_L_mixed_inf} and \Cref{lemma:small-gain-type-inequality}. 
    Next, observe that
    \begin{align}
       \norm*{(I - \mathbf{G}\mathbf{K}_r)^{-1}\mathbf{G} \mathbf{K}_r}_{\mathcal{H}_2} &= \norm*{(I - \mathbf{G}\mathbf{K}_r)^{-1}\mathbf{G} (\mathbf{K} + \mathbf{\Delta})}_{\mathcal{H}_2} \nonumber \\
       &\leq \norm*{(I - \mathbf{G}\mathbf{K}_r)^{-1}\mathbf{G} \mathbf{K}}_{\mathcal{H}_2} + \norm*{(I - \mathbf{G}\mathbf{K}_r)^{-1}\mathbf{G} \mathbf{\Delta}}_{\mathcal{H}_2} \nonumber \\
       &\leq \frac{1}{1 - \norm*{\mathbf{X}}_{\mathcal{H}_{\infty}} \norm*{\mathbf{\Delta}}_{\mathcal{H}_\infty}} \left(\norm*{(I - \mathbf{G}\mathbf{K})^{-1}\mathbf{G} \mathbf{K}}_{\mathcal{H}_2} \right) \nonumber \\
       &\qquad \qquad + \frac{\norm*{\mathbf{\Delta}}_{\mathcal{H}_{\infty}}}{1 - \norm*{\mathbf{X}}_{\mathcal{H}_{\infty}} \norm*{\mathbf{\Delta}}_{\mathcal{H}_\infty}} \left(\norm*{(I - \mathbf{G}\mathbf{K})^{-1}\mathbf{G}}_{\mathcal{H}_2} \right), \label{eq:eq:12_K_r_sub-block_H2_diff}
    \end{align}
    where we used \cref{eq:first_K_r_sub-block_H2_diff} and similar norm inequalities 
     to derive the last inequality.
We then consider the term $\mathbf{K}_r (I - \mathbf{G} \mathbf{K}_r)^{-1}$. From \cref{eq:Kr(I-GKr)_inv_expression} in the proof of \cref{theorem:truncated_K_still_stable_conditions}, we have 
    \begin{align*}
        \mathbf{K}_r (I - \mathbf{G} \mathbf{K}_r)^{-1} = \trK(I - \mathbf{G}\trK)^{-1} + \trK \trX  (I - \trDelta \mathbf{X})^{-1} \mathbf{\Delta}  (I - \mathbf{G} \mathbf{K})^{-1} + (I - \trDelta\mathbf{X} )^{-1}\trDelta (I - \mathbf{G}\trK)^{-1},
    \end{align*}
leading to the following upper bound
    \begin{align} \label{eq:upper-bound-3}
        \norm*{\mathbf{K}_r (I - \mathbf{G} \mathbf{K}_r)^{-1}}_{\mathcal{H}_2} 
        \leq \norm*{\trK(I - \mathbf{G}\trK)^{-1}}_{\mathcal{H}_2} + \norm*{\mathbf{\Delta}}_{\mathcal{H}_2} 
        \frac{\norm*{(I - \mathbf{G} \trK)^{-1}}_{\mathcal{H}_{\infty}} \left(1 + \norm*{\trK \trX}_{\mathcal{H}_{
        \infty}} \right)}{1 - \norm*{\mathbf{X}}_{\mathcal{H}_{\infty}} \norm*{\mathbf{\Delta}}_{\mathcal{H}_{\infty}}}. 
    \end{align}
    %
    %
    Similarly, we can derive the following upper bound
    \begin{align} \label{eq:upper-bound-4}
        \norm*{\mathbf{K}_r (I \!-\! \mathbf{G} \mathbf{K}_r)^{-1}\trG}_{\mathcal{H}_2} 
        \leq \norm*{\trK(I \!-\! \mathbf{G}\trK)^{-1}\trG}_{\mathcal{H}_2} + \norm*{\mathbf{\Delta}}_{\mathcal{H}_2} 
        \frac{\norm*{(I \!-\! \mathbf{G} \trK)^{-1}\trG}_{\mathcal{H}_{\infty}} \left(1 \!+\! \norm*{\trK \trX}_{\mathcal{H}_{
        \infty}} \right)}{1 - \norm*{\mathbf{X}}_{\mathcal{H}_{\infty}} \norm*{\mathbf{\Delta}}_{\mathcal{H}_{\infty}}}. 
    \end{align}

Combining the bounds in \cref{eq:first_K_r_sub-block_H2_diff}, \cref{eq:eq:12_K_r_sub-block_H2_diff}, \cref{eq:upper-bound-3} and  \cref{eq:upper-bound-4}, we arrive at the desired bound as follows 
\begin{align*}
J(\mathbf{K}_r) 
\leq  \left(\frac{1}{1 - \norm*{\mathbf{X}}_{\mathcal{H}_{\infty}} \norm*{\mathbf{\Delta}}_{\mathcal{H}_{\infty}}}\right)^2 (J(\mathbf{K}) + S_1 + S_2),
\end{align*}
where $S_1$ and $S_2$ are defined in \cref{eq:S1-S2} (recalling the notation $\mathbf{Y} = (I - \mathbf{G}\trK)^{-1},\mathbf{X} = (I - \mathbf{G}\trK)^{-1} \trG$ in \cref{eq:definition-X-Y}), and 
$$
    J(\mathbf{K}) \!=\! \norm*{(I \!-\! \mathbf{G}\mathbf{K})^{-1}\mathbf{G}}_{\mathcal{H}_2}^2 + \norm*{(I \!-\! \mathbf{G}\mathbf{K})^{-1}\mathbf{G} \trK}_{\mathcal{H}_2}^2 + \norm*{\trK(I \!-\! \mathbf{G}\mathbf{K})^{-1}\mathbf{G}}_{\mathcal{H}_2}^2 + \norm*{\trK(I \!-\! \mathbf{G}\mathbf{K})^{-1}}_{\mathcal{H}_2}^2. 
$$
This finishes the proof. 
\end{proof}

\Cref{theorem:stable_K_truncation_result} shows that as long as the truncation error is bounded as in \cref{eq:K_Kr_pole_stability_condition_1}, the reduced-order controller $\mathbf{K}_r$ still internally stabilizes the plant. Note that the bound in \cref{eq:K_Kr_pole_stability_condition_1} can be verified since $\mathbf{K}$ and $\mathbf{G}$ are known by the designer. Furthermore, the upper bound~\eqref{eq:LQG-cost-bound} implies that 
$$
\left|\frac{J(\mathbf{K}_r) - J(\mathbf{K})}{J(\mathbf{K})}\right| \leq \mathcal{O}(\norm*{\mathbf{\Delta}}_{\mathcal{H}_{\infty}}).
$$
When the truncation error is small (measured by $\norm*{\mathbf{\Delta}}_{\mathcal{H}_{\infty}}$), the change of LQG cost is also small. Similar bounds like \cref{eq:LQG-cost-bound} seem to be less studied for general non-observer based controllers in the literature \citep{anderson1987controller}. Most existing bounds assume an observed-based controller in \cref{eq:observer} (cf. \cite{liu1986controller,anderson1987controller,zhou1995performance}), and most of the techniques therein rely on coprime factorization \citep{vidyasagar1975coprime}. \Cref{theorem:truncated_K_still_stable_conditions} and \Cref{theorem:stable_K_truncation_result} work for any perturbed controller $\mathbf{K}_r$ satisfying the assumptions therein. In the next two sections, we show how to use balanced and modal truncation strategies to derive suitable reduced-order controllers $\mathbf{K}_r$.

\section{Controller reduction via balanced truncation}
\label{section:controller_reduction_via_balanced_truncation}
In this section, we discuss controller reduction strategies using balanced truncation and apply \Cref{theorem:truncated_K_still_stable_conditions} and \Cref{theorem:stable_K_truncation_result} to derive stability and performance guarantees. 
%
\subsection{Balanced truncation}
We begin with a result reviewing the following well-known fact about balanced truncation: for asymptotically stable transfer functions, under appropriate assumptions, a reduced-order transfer function resulting from balanced truncation is also asymptotically stable. 
\vspace{2mm}
\begin{lemma}[{\cite[Theorem 3.2]{pernebo1982model}, \cite[Theorem 7.9]{antoulas2005approximation}}]
\label{lemma:truncated_distinct_evals_assumption}
Let $\mathbf{P}$ be an asymptotically stable transfer function  (i.e. all poles are in the open left-half plane) with a balanced minimal state-space realization
\begin{align*}
    \mathbf{P} = \left[\begin{array}{c c |c}
        A_{11} & A_{12} & B_1 \\
        A_{21} & A_{22} & B_2 \\
        \hline
        C_1 & C_2 & D
    \end{array}\right], \mbox{ s.t. }&\begin{bmatrix}
    A_{11} & A_{12} \\
    A_{21} & A_{22}
    \end{bmatrix}\!\begin{bmatrix} 
    \Sigma_1 & 0 \\
    0 & \Sigma_2
    \end{bmatrix} \!+\! \begin{bmatrix}
    \Sigma_1 & 0 \\
    0 & \Sigma_2
    \end{bmatrix}\! \begin{bmatrix}
    A_{11} & A_{12} \\
    A_{21} & A_{22}
    \end{bmatrix}^\tr\! =\! -\begin{bmatrix}
    B_1 \\
    B_2
    \end{bmatrix}\!\begin{bmatrix}
    B_1 \\
    B_2
    \end{bmatrix}^\tr \\
    &\begin{bmatrix}
    \Sigma_1 & 0 \\
    0 & \Sigma_2
    \end{bmatrix} \!\begin{bmatrix}
    A_{11} & A_{12} \\
    A_{21} & A_{22}
    \end{bmatrix}\! +\! \begin{bmatrix}
    A_{11} & A_{12} \\
    A_{21} & A_{22}
    \end{bmatrix}^\tr \! \begin{bmatrix}
    \Sigma_1 & 0 \\
    0 & \Sigma_2
    \end{bmatrix} \!=\! -\begin{bmatrix}
    C_1^\tr \\ 
    C_2^\tr
    \end{bmatrix} \! \begin{bmatrix}
    C_1^\tr \\ 
    C_2^\tr
    \end{bmatrix}^\tr
\end{align*}
\normalsize
where $\Sigma_{i} \succ 0 $ is a positive-definite diagonal matrix for each $i \in \{1,2\}$. If $\Sigma_1$ and $\Sigma_2$ share no eigenvalues in common, then the sub-blocks $A_{11}$ and $A_{22}$ are both asymptotically stable. Furthermore, we have the following error bound
\begin{equation} \label{eq:balanced-error}
\|\mathbf{P} - \mathbf{P}_r\|_{\mathcal{H}_\infty} \leq 2\mathrm{trace}(\Sigma_2), \; \text{where}\; \mathbf{P}_r \!=\! \left[\begin{array}{c |c}
        A_{11} &B_1 \\
        \hline
        C_1  & D
    \end{array}\right].
\end{equation}
\end{lemma}

Note that \cref{lemma:truncated_distinct_evals_assumption} assumes that the transfer function $\mathbf{P}$ has a balanced state-space realization. It is well-known from classical control that any stable transfer function has a balanced minimal realization. For completeness, we state (and prove) this result in \cref{prop:balanced_realization} in \cref{subsec:balanced_realization}. In addition, in \Cref{algorithm:balanced_truncation}, we state the standard balanced truncation procedure for a minimal realization of a stable controller $\mathbf{K}$.

\begin{algorithm}[tp!]
\caption{Balanced truncation for stable systems}
\label{algorithm:balanced_truncation}
\begin{algorithmic}[1]
\REQUIRE{
1) A stable system $\trK$ with minimal state-space  realization $\trK = \left[\begin{array}{c|c}
    A & B \\
    \hline
    C & D
\end{array} \right]$, 2) post-truncation order parameter $r$
}
\STATE{Form the controllability and observability Gramians $W_c$ and $W_o$ via solving 
\begin{align*}
    {A} W_c + W_c {A}^\tr + {B} {B}^\tr = 0, \qquad
     {A}^\tr W_o + W_o {A} + {C}^\tr {C} = 0. 
\end{align*}}
\vspace{-4mm}
\STATE{Compute $\Sigma = \diag\left(\left\{\sqrt{\lambda_i(W_cW_o)}\right\}_{i=1}^n\right)$.}
\STATE{Factorize $W_c = QQ^\tr$, 
and compute an orthonormal $U$ such that $Q^\tr W_o Q = U\Sigma^2 U^\tr$.}
\STATE{Let $T = \Sigma^{1/2}U^\tr Q^{-1}$, and compute a balanced minimal realization $\trK = \left[\begin{array}{c|c}
    \tilde{A} & \tilde{B}  \\
    \hline \\[-10pt]
    \tilde{C} & \tilde{D} 
\end{array}\right],$ where $\tilde{A} =  TAT^{-1}, \tilde{B} = TB, \tilde{C} = CT^{-1}, \tilde{D} = D.$}
\STATE{Partition $\tilde{A}, \tilde{B}, \tilde{C}, \Sigma$ as 
 $  \tilde{A} = \begin{bmatrix}
    \tilde{A}_{11} & \tilde{A}_{12} \\
    \tilde{A}_{21} & \tilde{A}_{22}
    \end{bmatrix}, 
    \tilde{B} = \begin{bmatrix}
    \tilde{B}_{1} \\
    \tilde{B}_{2}
    \end{bmatrix},  
    \tilde{C} = \begin{bmatrix}
    \tilde{C}_{1} &
    \tilde{C}_{2}
    \end{bmatrix},
    \Sigma = \begin{bmatrix}
    \Sigma_1 & 0 \\
    0 & \Sigma_2
    \end{bmatrix},$
where $\tilde{A}_{11}, \Sigma_1 \in \bbR^{r \times r}, \tilde{B}_1 \in \bbR^{r \times p}, \tilde{C}_1 \in \bbR^{m \times r}$.}
\RETURN{the reduced order controller 
$\trK_r \coloneqq \left[\begin{array}{c|c}
    \tilde{A}_{11} & \tilde{B}_1 \\
    \hline \\[-10pt]
    \tilde{C}_{1} & \tilde{D}
\end{array} \right].$}
\end{algorithmic}
\end{algorithm}

\subsection{Controller reduction} \label{subsection:stable-unstable-separation}
In general, the dynamical controller $\trK$ is not stable itself, i.e., $A_\mK$ in~\cref{eq:Dynamic_Controller} has unstable eigenvalues. The standard balanced truncation procedure cannot be applied to unstable systems directly. Our strategy is to divide the controller $\trK$ into a stable part and unstable part
\begin{equation} \label{eq:controller-separation}
    \trK = \trK_< + \trK_\geq,
\end{equation}
\normalsize
where $\trK_<$ of order $n_1$ contains all stable poles (i.e., those on the open left-half plane) and $\trK_\geq$ of order $n_2$ contains the remaining poles (i.e., those on the closed right-half plane), and $n_1 + n_2 = n$. 
In this section, we assume the controller contains at least one stable pole ($n_1 \geq 1$). 

The separation \cref{eq:controller-separation} is always possible by computing the Jordan normal form of $A_\mK$ such that  
\begin{equation} \label{eq:Jordan-normal-form}
A_\mK = Q_\mK \hat{A}_\mK Q_\mK^{-1}, \qquad \text{with} \;\; \hat{A}_\mK = \begin{bmatrix}
\hat{A}_{\mK,<} & 0 \\
0 & \hat{A}_{\mK, \geq}
\end{bmatrix},
\end{equation}
\normalsize
where $Q_\mK \in \mathbb{R}^{n \times n}$ is an invertible coordinate transformation\footnote{Since any real-valued matrix can be expressed in a Jordan canonical form, such a transformation $Q_\mK$ always exists.},  the eigenvalues of $\hat{A}_{\mK,<} \in \bbR^{n_1 \times n_1}$ are in the open left-half plane, 
and the eigenvalues of $\hat{A}_{\mK,\geq} \in \bbR^{n_2 \times n_2}$ are in the closed right-half plane, and $n_1 + n_2 = n$. Therefore, the stable and unstable parts in \cref{eq:controller-separation} can be expressed as 
\begin{equation} \label{eq:controller-separation-state-space}
\trK_< =  \left[\begin{array}{c|c} \hat{A}_{\mK,<} & \hat{B}_{\mK,<} \\\hline
     \hat{C}_{\mK,<} & 0\end{array}\right], \qquad \trK_\geq = \left[\begin{array}{c|c} \hat{A}_{\mK,\geq} & \hat{B}_{\mK,\geq} \\\hline
     \hat{C}_{\mK,\geq} & 0\end{array}\right],
\end{equation}
\normalsize
where $\hat{C}_\mK \coloneqq C_\mK Q_\mK$ and $\hat{B}_\mK \coloneqq Q_\mK^{-1} B_\mK$ are partitioned into 
$$
\hat{C}_\mK = \begin{bmatrix}
\hat{C}_{\mK,<} & \hat{C}_{\mK,\geq} 
\end{bmatrix}, \qquad \hat{B}_\mK = \begin{bmatrix}
\hat{B}_{\mK,<} \\ \hat{B}_{\mK,\geq} 
\end{bmatrix}
$$
\normalsize
with $\hat{C}_{\mK,<} \in \bbR^{m \times n_1}, \hat{C}_{\mK,\geq} \in \bbR^{m \times (n-n_1)}$, and  $\hat{B}_{\mK,<} \in \bbR^{n_1 \times p}, \hat{B}_{\mK,\geq} \in \bbR^{(n - n_1) \times p}$. 

We can then perform a balanced truncation on the stable part $\trK_<$ and get a reduced-order controller 
$$
\trK_{<,{r}} =  \left[\begin{array}{c|c} \tilde{A}_{<,11} & \tilde{B}_{<,1} \\\hline
\\[-12pt]
     \tilde{C}_{<,1} & 0\end{array}\right],
$$
\normalsize
where the order is $n_r < n_1$. The final reduced-order controller becomes
\begin{equation}
    \trK_r = \trK_{<,r} + \trK_{\geq} = \left[\begin{array}{cc|c} \tilde{A}_{<,11} & 0 & \tilde{B}_{<,1} \\ 
    0 & \hat{A}_{\mK,\geq} & \hat{B}_{\mK,\geq}
    \\\hline
    \\[-12pt]
   \tilde{C}_{<,1} & \hat{C}_{\mK,\geq}& 0\end{array}\right]
\end{equation}
\normalsize
which has order $r:= n_r + n_2 < n$. This process is summarized in \Cref{algorithm:balanced_truncation_stable_unstable}. 

\begin{algorithm}[tp!]
\caption{Balanced truncation for an unstable controller with stable part}
\label{algorithm:balanced_truncation_stable_unstable}
\begin{algorithmic}[1]
\REQUIRE{1) A controller $\trK$ with a minimal order-$n$ state-space  realization  $\trK = \left[\begin{array}{c|c}
    A_\mK & B_\mK \\
    \hline
    C_\mK & 0
\end{array} \right]$, 2) the post-truncation order $r\geq n_2$ with $n_2$ defined in \cref{eq:controller-separation}.}
\STATE{Compute the Jordan normal form of $A_\mK$ in \cref{eq:Jordan-normal-form}.}
\STATE{Separate the controller $\trK$ into  $\trK = \trK_< + \trK_\geq$ as \cref{eq:controller-separation-state-space}.} 
\STATE{Perform balanced truncation on the stable and minimal system $\trK_< $ with post-truncation order parameter $n_r < n_1$ to obtain an order-$n_r$ system $\trK_{<,r}$.}
\RETURN{the reduced-order controller $\trK_r = \trK_{r,<} + \trK_{\geq}$ (of order $r=n_r + n_2 < n)$. }
\end{algorithmic}
\end{algorithm}

Based on \Cref{theorem:truncated_K_still_stable_conditions,theorem:stable_K_truncation_result},
under appropriate conditions, the resulting controller $\trK_r$ from \Cref{algorithm:balanced_truncation_stable_unstable} remains a stabilizing controller and has a similar LQG cost compared to the original controller $\trK$.


\vspace{1mm}

\begin{corollary}
\label{corollary:balanced_truncation_LQG_bound}
Consider a minimal $n$-th order controller $\mathbf{K}$ which stabilizes the plant $\trG$. 
Suppose we obtain an $r$-th order controller $\trK_r$ via \Cref{algorithm:balanced_truncation_stable_unstable}, where $r < n$, such that $\trK_r = \trK_{<,r} + \trK_{\geq}$, where $\trK_{<,r}$ is a lower-order balanced truncation of $\trK_<$. Suppose that $\Sigma_{<,1}$ and $\Sigma_{<,2}$ in the balanced truncation of $\trK_<$ share no eigenvalues. If
\begin{align}
\label{eq:K_Kr_pole_stability_condition_2}
    \sigma_{n_r+1} + \cdots +\sigma_{n_1}< \frac{1}{ 2\norm*{(I - \mathbf{G}\mathbf{K})^{-1} \mathbf{G}}_{\mathcal{H}_\infty}},
\end{align}
where $\sigma_{n_r+1},\dots,\sigma_{n_1}$ are the diagonal elements of $\Sigma_{<,2}$ in the balanced truncation of $\trK_{<}$, then the reduced-order controller $\trK_r$ internally stabilizes the plant \cref{eq:LTI}, and the resulting LQG cost  \cref{eq:LQG-cost}  satisfies 
    $$J(\trK_r) \leq \frac{1}{\left(1 - \norm*{\mathbf{X}}_{\mathcal{H}_{\infty}} \norm*{\mathbf{\Delta}}_{\mathcal{H}_{\infty}}\right)^2} (J(\mathbf{K}) + S_1 + S_2),$$
where with the notation $\mathbf{Y} := (I - \mathbf{G}\trK)^{-1}$, $\mathbf{X} := (I - \mathbf{G}\trK)^{-1} \mathbf{G}$ and $\mathbf{\Delta} := \mathbf{K}_r - \mathbf{K}$,
\begin{align*}
    &S_1 := 2\norm*{\mathbf{\Delta}}_{\mathcal{H}_\infty} \norm*{\mathbf{X}}_{\mathcal{H}_2}(\norm*{\mathbf{X} \mathbf{K}}_{\mathcal{H}_2})
    + 2\norm*{\mathbf{\Delta}}_{\mathcal{H}_2} (\norm*{\trK \mathbf{Y}}_{\mathcal{H}_2} \norm*{\mathbf{Y}}_{\mathcal{H}_\infty} + \norm*{\trK \mathbf{X}}_{\mathcal{H}_2} \norm*{\mathbf{X}}_{\mathcal{H}_\infty})(1 + \norm*{\trK \trX}_{\mathcal{H}_{
        \infty}} ), \\
    &S_2 := \norm*{\mathbf{\Delta}}_{\mathcal{H}_\infty}^2 \norm*{\mathbf{X}}_{\mathcal{H}_2}^2 + \norm*{\mathbf{\Delta}}_{\mathcal{H}_2}^2 (\norm*{\mathbf{Y}}_{\mathcal{H}_{\infty}}^2 + \norm*{\mathbf{X}}_{\mathcal{H}_{\infty}}^2) (1 + \norm*{\trK \trX}_{\mathcal{H}_{
        \infty}} )^2.
\end{align*}
\normalsize
\end{corollary}

The proof is based on the analysis in \Cref{theorem:stable_K_truncation_result}. The only difference is that \Cref{corollary:balanced_truncation_LQG_bound} truncates the singular values in the stable part $\trK_r$ and imposes the condition \cref{eq:K_Kr_pole_stability_condition_2}, which is the same \cref{eq:K_Kr_pole_stability_condition_1} in \Cref{theorem:stable_K_truncation_result} when applying \cref{eq:balanced-error} in \Cref{lemma:truncated_distinct_evals_assumption}. Note that \Cref{corollary:balanced_truncation_LQG_bound} is reduced to \Cref{theorem:stable_K_truncation_result} when the controller $\mathbf{K}$ is stable itself.

\section{Controller reduction via modal truncation} \label{section:order_reduction_modal}


In this section, we proceed to discuss controller reduction by modal truncation, which may apply to the truncation of either stable or unstable component(s) in a controller.  
In particular, we first apply modal truncation on the stable part of a controller in \cref{section:controller_reduction_modal_truncation_stable}, and then discuss the performance of modal truncation on possibly unstable component(s) for SISO systems in \cref{section:order_reduction_for_unstable_controller_SISO}. 

\subsection{Modal truncation on stable component(s)}
\label{section:controller_reduction_modal_truncation_stable}

\begin{algorithm}[tp!]
\caption{Modal truncation}
\label{algorithm:modal_truncation}
\begin{algorithmic}[1]
\REQUIRE{
1) A controller $\trK$ with a minimal order-$n$ state-space  realization  $\trK = \left[\begin{array}{c|c}
    A_\mK & B_\mK \\
    \hline
    C_\mK & 0
\end{array} \right]$; 2) the order~reduction parameter $r_{\mathrm{red}}$ (a positive integer less than $k$)}
\STATE{Convert $A_\mK$ into the standard Jordan normal form \eqref{eq:Jordan-form}. }
\STATE{Decompose $\trK$ as $\trK(s) = \sum_{i=1}^k C_i (sI-A_i)^{-1} B_i$ that is consistent with the Jordan block \cref{eq:Jordan-form}.} 
\STATE{For each $i$, compute the index $d_i$ in \eqref{eq:modes-importance}.}  
\STATE{Let $o_i$ be the ranking of $i$ according to $\{d_i\}_{i=1}^k$, such that $o_i = j$ if $d_i$ is the $j$-th smallest value. 
}
\STATE{Set $\trDelta = \sum_{i \in [k], o_i \leq r_{\mathrm{red}}} C_i (sI - A_i)^{-1} B_i$.}
\RETURN{the reduced order controller $\trK_r := \trK - \trDelta$.}
\end{algorithmic}
\end{algorithm}

The basic idea of modal truncation begins with writing the controller $\trK$ \cref{eq:Dynamic_Controller} into $\trK(s) = \sum_{i=1}^k C_i (sI-A_i)^{-1} B_i$, where $A_i$ contains a mode corresponding to an eigenvalue of $\lambda_i$ in $\trK$. This is always possible by considering its standard Jordan form 
\begin{equation} \label{eq:Jordan-form}
A_\mK =  \begin{bmatrix}
    A_1 & 0 & \dots & 0 \\
    0 & A_2 & \dots & 0 \\
    \vdots & \vdots & \ddots & \vdots \\
    0 & 0 & \dots & A_k
    \end{bmatrix},
\end{equation}
\normalsize
where each $A_i$ is a Jordan block of order $n_i$, and $\sum_{i=1}^k n_i = n.$ Let $\lambda_i$ denote the eigenvalue associated with each Jordan block $A_i$. 
We then directly remove some modes that are less significant according the criterion defined below 
\begin{equation} \label{eq:modes-importance}
     d_i =
\begin{cases}
\norm*{C_i (sI - A_i)^{-1} B_i}_{\mathcal{H}_\infty} & \text{if} \; \lambda_i < 0  \\
\norm*{C_i A_i^{-1} B_i}_2  & \text{if} \; \lambda_i > 0.
\end{cases} 
\end{equation}
The detailed steps are listed in \Cref{algorithm:modal_truncation}.     

As a counterpart to balanced truncation, we can derive upper bounds the LQG cost change when performing modal truncation on the stable part of a controller $\trK$. 

\vspace{1mm}
\begin{corollary}
\label{corollary:modal_truncation_stable_unstable}
Consider a minimal $n$-th order controller $\mathbf{K}$ which stabilizes the plant $\trG$. 
Consider the decomposition $\trK = \trK_< + \trK_\geq$ in \cref{eq:controller-separation}, and suppose we obtain a lower-order approximation $\trK_{r,<}$ of $\trK_<$ using the modal truncation algorithm in \Cref{algorithm:modal_truncation}. Let $\trK_r := \trK_{r,<} + \trK_{\geq}$. 
Denote $\mathbf{\Delta} = \trK_< - \trK_{r,<}$. Suppose that
\vspace{-1mm}
\begin{align}
\label{eq:K_Kr_pole_stability_condition_stable}
 \norm*{\trDelta}_{\calH_{\infty}} < \frac{1}{\norm*{(I - \mathbf{G}\mathbf{K})^{-1} \mathbf{G} }_{\mathcal{H}_{\infty}}}.
\end{align}
\normalsize
Then, we have that $\trK_r$ internally stabilizes $\trG$, and the resulting LQG cost \cref{eq:LQG-cost} satisfies
\vspace{-1mm}
$$J(\trK_r) \leq \frac{1}{\left(1 - \norm*{\mathbf{X}}_{\mathcal{H}_{\infty}} \norm*{\mathbf{\Delta}}_{\mathcal{H}_{\infty}}\right)^2} (J(\mathbf{K}) + S_1 + S_2),$$
\normalsize
where with the notation $\mathbf{Y} := (I - \mathbf{G}\trK)^{-1}$, $\mathbf{X} := (I - \mathbf{G}\trK)^{-1} \mathbf{G}$, we have
\begin{align*}
     &S_1 := 2\norm*{\mathbf{\Delta}}_{\mathcal{H}_\infty} \norm*{\mathbf{X}}_{\mathcal{H}_2}(\norm*{\mathbf{X} \mathbf{K}}_{\mathcal{H}_2}) 
    + \norm*{\mathbf{\Delta}}_{\mathcal{H}_2} (\norm*{\trK \mathbf{Y}}_{\mathcal{H}_2} \norm*{\mathbf{Y}}_{\mathcal{H}_\infty} + \norm*{\trK \mathbf{X}}_{\mathcal{H}_2} \norm*{\mathbf{X}}_{\mathcal{H}_\infty})\left(1 + \norm*{\trK \trX}_{\mathcal{H}_{
        \infty}} \right), \\
    &S_2 := \norm*{\mathbf{\Delta}}_{\mathcal{H}_\infty}^2 \norm*{\mathbf{X}}_{\mathcal{H}_2}^2 + \norm*{\mathbf{\Delta}}_{\mathcal{H}_2}^2 (\norm*{\mathbf{Y}}_{\mathcal{H}_{\infty}}^2 + \norm*{\mathbf{X}}_{\mathcal{H}_{\infty}}^2) \left(1 + \norm*{\trK \trX}_{\mathcal{H}_{
        \infty}} \right)^2.
\end{align*}
\normalsize
\end{corollary}
\vspace{-3mm}
\begin{proof}
Since \cref{eq:K_Kr_pole_stability_condition_stable} holds, the condition in \cref{eq:truncated_K_still_stable_conditions} holds, and we can thus apply \cref{theorem:truncated_K_still_stable_conditions} to conclude that $\trK_r$ internally stabilizes $\trG$. Then, the same calculations in \cref{theorem:stable_K_truncation_result} can be used to show the upper bound on $J(\trK_r)$.
\end{proof}

\subsection{Modal truncation on unstable component(s)}
\label{section:order_reduction_for_unstable_controller_SISO}

Next, we introduce the following result, which studies the LQG cost change when truncating unstable mode(s) of a controller $\trK$, for single-input single-output (SISO) systems.


\begin{theorem}[Order reduction of unstable SISO controllers]
\label{theorem:modal_truncation_SISO}
Consider a minimal $n$-th order controller $\mathbf{K}$ which stabilizes the plant $\trG$. 
Suppose both $\trG(s)$ and $\trK(s)$ are univariate rational polynomial functions (SISO systems). Let an $r$-order controller $\trK_r$ computed via \Cref{algorithm:modal_truncation}, where $r < n$, where we denote $\mathbf{\Delta} = \trK - \trK_{r}$ as in \Cref{algorithm:modal_truncation} (note that $\trDelta$ may be unstable). Suppose $\trDelta$ has no unstable mode at 0, and 
\begin{align}
\label{eq:K_Kr_pole_stability_condition_unstable}
(1 - (1 - \mathbf{G}\mathbf{K})^{-1} \mathbf{G} \trDelta)^{-1} \in \mathcal{RH}_\infty,
\end{align}
\normalsize
Then, $\trK_r$ internally stabilizes $\trG$, and 
\begin{equation} \label{eq:LQG-unstable-SISO}
J(\trK_r) \leq \norm*{(1 - \trX \trDelta)^{-1}}_{\mathcal{H}_{\infty}}^2 (J(\mathbf{K}) + S_1 + S_2),
\end{equation}
\normalsize
where with the notation $\mathbf{Y} := (I - \mathbf{G}\trK)^{-1}$, $\mathbf{X} := (I - \mathbf{G}\trK)^{-1} \mathbf{G}$, we have
\begin{align*}
    &S_1\! := \!2\norm*{\mathbf{\Delta}}_{\mathcal{L}_\infty} \norm*{\mathbf{X}}_{\mathcal{H}_2}\norm*{\mathbf{X} \mathbf{K}}_{\mathcal{H}_2} +2\!\norm*{\trK \mathbf{Y}}_{\mathcal{H}_2}\! \left(\norm*{\mathbf{\Delta}}_{\mathcal{L}_2}\norm*{\mathbf{Y}}_{\mathcal{H}_\infty}\right) \\
    &\quad  \quad  
    +2\norm*{\trK \mathbf{Y}}_{\mathcal{H}_2}\! \left(\!\norm*{\mathbf{\Delta}}_{\mathcal{L}_\infty}\!\left(\norm*{\mathbf{Y}}_{\mathcal{H}_2}\! +\!\norm*{\trK \mathbf{X}}_{\mathcal{H}_2}\norm*{\mathbf{Y}}_{\mathcal{H}_\infty}\! \!+\! \norm*{\trK \mathbf{X}}_{\mathcal{H}_\infty}\norm*{\mathbf{X}}_{\mathcal{H}_2} \!\right)\!\right) \\
&S_2 \!:= \!\norm*{\mathbf{\Delta}}_{\mathcal{L}_\infty}^2 \norm*{\mathbf{X}}_{\mathcal{H}_2}^2 
+ \left(\norm*{\mathbf{Y}}_{\mathcal{H}_\infty} \left(\norm*{\trDelta}_{\mathcal{L}_\infty} \norm*{\trK \mathbf{X}}_{\mathcal{H}_2}  +  \norm*{\trDelta}_{\mathcal{L}_2}\right) \right)^2 
    + \norm*{\mathbf{X}}_{\mathcal{H}_2}^2\left(\norm*{\trDelta}_{\mathcal{L}_\infty} \norm*{\trK \mathbf{X}}_{\mathcal{H}_\infty}  +  \norm*{\trDelta}_{\mathcal{L}_\infty}\right)^2.
\end{align*}
\normalsize
\end{theorem}

\begin{proof}
Without loss of generality, we suppose $\trDelta(s) = C_k (sI - A_k)^{-1} B_k$ and assume that $\lambda_k > 0$. From the proof of \cref{theorem:truncated_K_still_stable_conditions}, to show that $\trK_r$ internally stabilize $\trG$, it suffices for us to show that
\begin{align*}
\begin{bmatrix}
        (I - \mathbf{G}\mathbf{K}_r)^{-1} \mathbf{G} & (I - \mathbf{G}\mathbf{K}_r)^{-1} \mathbf{G} \mathbf{K}_r \\
    \mathbf{K}_r(I - \mathbf{G}\mathbf{K}_r)^{-1} \mathbf{G} &  \mathbf{K}_r (I - \mathbf{G}\mathbf{K}_r)^{-1}
    \end{bmatrix} \in \mathcal{RH}_\infty. 
\end{align*}
From \cref{eq:I-GK_r_inv_G_expression} in the proof of \cref{theorem:truncated_K_still_stable_conditions}, since $\trK$ stabilizes $\trG$, the stability of $(I - \mathbf{G}\mathbf{K}_r)^{-1} \mathbf{G}$ is ensured when $(I - \trX \trDelta)^{-1}$ is stable. 

Meanwhile, from \cref{eq:I-GK_r_inv_G_Kr_expression} in the proof of \cref{theorem:truncated_K_still_stable_conditions}, assuming $\trK$ stabilizes $\trG$, the stability of $(I - \mathbf{G}\mathbf{K}_r)^{-1} \mathbf{G} \trK_r$ is achieved when $(I - \trX \trDelta)^{-1}$ and $\trX \trDelta$ are stable. 

Next, to show that $\mathbf{K}_r (I - \mathbf{G}\mathbf{K}_r)^{-1}$ is stable, from \cref{eq:Kr(I-GKr)_inv_expression} in the proof of \cref{theorem:truncated_K_still_stable_conditions}, assuming $\trK$ stabilizes $\trG$, it suffices for us to show that $(I - \trDelta \trX )^{-1}$ and $\trDelta \trY$ are stable. Observe that in the SISO case,  $(I - \trX \trDelta )^{-1} = (I - \trDelta \trX )^{-1}$.

Note that in the SISO case, since $(1 - \mathbf{G}\mathbf{K}_r)^{-1} \mathbf{G} \mathbf{K}_r = \mathbf{K}_r(1 - \mathbf{G}\mathbf{K}_r)^{-1} \mathbf{G}$, we do not have to separately show the stability of $\mathbf{K}_r(1 - \mathbf{G}\mathbf{K}_r)^{-1} \mathbf{G}$. 

Thus, collecting the conditions above, given that $\trK$ internally stabilizes $\trG$, to show that $\trK_r$ internally stabilizes $\trG$, recalling that $\trX := (1 - \mathbf{G}\mathbf{K})^{-1} \trG$ and $\trY :=(1 - \mathbf{G}\mathbf{K})^{-1} $, it suffices for us to show that 
\begin{align*}
    (1 - (1 - \mathbf{G}\mathbf{K})^{-1} \trG \trDelta)^{-1} \in \mathcal{RH}_\infty, \quad (1 - \mathbf{G}\mathbf{K})^{-1} \trG \trDelta \in \mathcal{RH}_\infty, \quad \trDelta(1 - \trG \trK)^{-1} \in \mathcal{RH}_\infty.
\end{align*}
Our assumption in \cref{eq:K_Kr_pole_stability_condition_unstable} assumes that $(1 - (1 - \mathbf{G}\mathbf{K})^{-1} \trG \trDelta)^{-1} \in \mathcal{RH}_\infty$. We proceed to show now that (with the aid of the assumption that $(1 - (1 - \mathbf{G}\mathbf{K})^{-1} \trG \trDelta)^{-1} \in \mathcal{RH}_\infty$) 
\begin{align*}
(1 - \mathbf{G}\mathbf{K})^{-1} \trG \trDelta \in \mathcal{RH}_\infty, \quad \trDelta(1 - \trG \trK)^{-1} \in \mathcal{RH}_\infty.
\end{align*}




We first show that  $(1 - \mathbf{G}\mathbf{K})^{-1} \trG \trDelta \in \mathcal{RH}_\infty$. Since $(1 - \trG \trK)^{-1} \trG $ is stable, the only possible pole for $(1 - \trG \trK)^{-1} \trG \trDelta$ is at $s = \lambda_k$, where $\lambda_k$ is the eigenvalue associated with $A_k$. 
Thus, it suffices to show that $\abs*{(1 - \trG \trK)^{-1} \trG \trDelta(\lambda_k)} < \infty$.  
Observe that for any $i \in [k]$, the transfer function $C_i(sI - A_i)^{-1}B_i$ can be written as a rational function in the form
\begin{align*}
    C_i(sI - A_i)^{-1}B_i = \frac{\alpha_i(s)}{(s - \lambda_i)^{r_i}},
\end{align*}
where $r_i = n_i$ (this is due to the minimality of the state-space realization $(A_{\mK}, B_{\mK},C_{\mK},0)$),
and $\alpha_i(s)$ is coprime to $(s - \lambda_i)$. Then, we can express $\trK(s)$ as 
\begin{align*}
    \trK(s) = \sum_{i=1}^k C_i(sI - A_i)^{-1}B_i 
    = \sum_{i=1}^k \frac{\alpha_i(s)}{(s - \lambda_i)^{n_i}} 
    = \frac{\alpha(s)}{\prod_{i=1}^k (s - \lambda_i)^{n_i}},
\end{align*}
where $\alpha(s) = \sum_{i=1}^k \left(\alpha_i(s) \prod_{j \neq i} (s - \lambda_j)^{n_j}\right)$. 
Since this $(A_{\mK}, B_{\mK},C_{\mK},0)$ state-space realization of $\trK$ is minimal, $\alpha(s)$ is coprime to $(s - \lambda_i)$ for each $i \in [k]$. 


Note that it is possible that $\lambda_i = \lambda_k$ for $i \neq k$. Hence, by extracting the factors of $(s - \lambda_k)$, we may decompose
$$\prod_{i=1}^{k-1} (s - \lambda_i)^{n_i} = (s - \lambda_k)^{n_k'} \prod_{i=1, \lambda_i \neq \lambda_k}^{k-1} (s - \lambda_i)^{n_i},$$
where $n_k' = \sum_{i=1, \lambda_i = \lambda_k}^{k-1} n_i$. Then, we have
\begin{align}
    \frac{\trG(s) \trDelta(s)}{1 - \trG(s) \trK(s)} &= \frac{\trG(s) \frac{\alpha_k(s)}{(s- \lambda_k)^{n_k}}}{1 - \trG(s) \frac{\alpha(s)}{\prod_{i=1}^k (s - \lambda_i)^{n_i}}} \nonumber \\
    &= \frac{\alpha_k(s) (s - \lambda_k)^{n_k'}}{\frac{(s - \lambda_k)^{n_k + n_k'}}{\trG(s)} + \frac{\alpha(s)}{\prod_{i=1, \lambda_i \neq \lambda_k}^{k-1} (s - \lambda_i)^{n_i}}}. \label{eq:gd_times_1-gk_inv_calculation_intermediate}
\end{align}
Now, since $\trK$ internally stabilizes $\trG$, there is no unstable pole-zero cancellation between $\trG$ and $\trK$ \cite[Theorem 5.7]{zhou1996robust}. Thus, if we express $\trG(s)$ in rational polynomial form as $\trG(s) = \frac{g_n(s)}{g_d(s)}$, then $g_n(s)$ is coprime to $(s - \lambda_k)$ (recall that $\lambda_k$ is in the open right-half plane, by our assumption at the start of the proof). Thus, the polynomial
$$\frac{(s - \lambda_k)^{n_k + n_k'}}{\trG(s)}$$
evaluated at $s = \lambda_k$ is zero. Hence, continuing from \cref{eq:gd_times_1-gk_inv_calculation_intermediate}, we have
\begin{align*}
    \frac{\trG(\lambda_k) \trDelta(\lambda_k)}{1 - \trG(\lambda_k) \trK(\lambda_k)} 
    &= \frac{\alpha_k(\lambda_k) (\lambda_k - \lambda_k)^{n_k'}}{\frac{(\lambda_k - \lambda_k)^{n_k + n_k'}}{\trG(\lambda_k)} + \frac{\alpha(\lambda_k)}{\prod_{i=1, \lambda_i \neq \lambda_k}^{k-1} (\lambda_k - \lambda_i)^{n_i}}} \\
    &= \frac{\alpha_k(\lambda_k) (\lambda_k - \lambda_k)^{n_k'}}{\frac{\alpha(\lambda_k)}{\prod_{i=1, \lambda_i \neq \lambda_k}^{k-1} (\lambda_k - \lambda_i)^{n_i}}},
\end{align*}
which is finite, since $\alpha_k(s)$ and $\alpha(s)$ are both coprime to $(s - \lambda_k)$ and hence their pointwise evaluation at $s = \lambda_k$ is finite. Thus, we conclude that $\frac{\trG \trDelta}{1 - \trG \trK}$ does not have a pole at $\lambda_k$. 
This implies that it has a finite $\mathcal{H}_\infty$ norm. 

Next we proceed to show that 
$\trDelta (1 - \trG \trK)^{-1}\in \mathcal{RH}_\infty$. The argument for this is similar to the one above. Recall that we can write 
$$\trG(s) = \frac{g_n(s)}{g_d(s)},$$ where $g_n(s)$ is coprime to $(s - \lambda_k)$ and $g_n(s)$ and $g_d(s)$ are coprime. In particular, we can also decompose $g_d(s)$ as $(s - \lambda_k)^{r_g} g_d'(s)$, where $r_g$ is a nonnegative integer. Then, 
\begin{align*}
    \frac{\trDelta(s)}{1  - \trG(s) \trK(s)} = \frac{\frac{\alpha_k(s)}{(s - \lambda_k)^{n_k}}}{1 - \frac{g_n(s)}{g_d(s)} \frac{\alpha(s)}{\prod_{i=1}^k (s - \lambda_i)^{n_i}}} 
    = \frac{\alpha_k(s) (s - \lambda_k)^{n_k' + r_g}}{(s - \lambda_k)^{n_k + n_k' + r_g} + \frac{g_n(s)}{g_d'(s)}\frac{\alpha(s)}{\prod_{i=1, \lambda_i \neq \lambda_k}^{k-1} (s - \lambda_i)^{n_i}}} 
\end{align*}

From the above expression, since $g_n(s), g_d'(s)$ and $\alpha(s)$ are all coprime to $(s - \lambda_k)$, it is not hard to see that $\frac{\trDelta(\lambda_k)}{1 - \trG(\lambda_k) \trK(\lambda_k)}$ is finite (and must in fact be zero if $n_k' + r_g > 0$). Recall that since $\frac{1}{1 - \trG \trK}$ is stable (as $\trK$ internally stabilizes $\trG$), the only possible pole for $\frac{\trDelta}{1 - \trG \trK}$ is at a pole of $\trDelta$, which is at $s = \lambda_k$. Since we have just excluded this possibility, it follows that $\frac{\trDelta}{1 - \trG \trK}$ is stable.

The proof for the bound on the change in LQG performance \cref{eq:LQG-unstable-SISO} is similar to that in \Cref{theorem:stable_K_truncation_result}. We provide the details in \Cref{appendix:proof-LQG-bound}.  
\end{proof}

\begin{remark}
We note a limitation of our result, namely the requirement that the truncated Jordan blocks have non-zero eigenvalues. Hence, the procedure does not work when we wish to truncate a Jordan block corresponding to an zero eigenvalue. In addition, due to the relative simplicity of defining zeros for SISO systems, we chose to limit our attention to SISO systems. Extending to general MIMO remains future work. 
\end{remark}

\section{Connecting near pole-zero cancellation to small Jordan block}
\label{section:near-pole-zero cancellation_to_jordan block}

An intuitive way of defining ``near non-minimality'' for a transfer function $\trG(s)$ is the existence of a pair of pole $p_i$ and zero $q_i$ which are ``close'' to each other. Assuming that $p_i$ is a simple pole, when this happens, we then conjecture that the coefficient corresponding to the term $\frac{1}{s-p_i}$ in the partial fraction decomposition of $\trG(s)$ is small, i.e. the Jordan block corresponding to the pole $p_i$ is small.

Below, we formalize this idea in the case $p_i$ is a simple pole\footnote{We believe a similar result holds for the general case when $p_i$ is a repeated pole, and leave the precise characterization to future work.}.

\begin{lemma}
Consider a minimal transfer SISO transfer function $\trG(s)$. Suppose $p_i$ is a simple pole of $G(s)$. Suppose we have the factorization
\begin{align*}
    \trG(s) = \frac{s-q_i}{s-p_i}\frac{n(s)}{d(s)},
\end{align*}
where $n(s)$ is coprime to $d(s)$ and $s-p_i$, and $d(s)$ is also coprime to $s-p_i$. By Bezout's identity, we know there exists $a(s), b(s)$ such that $a(s) d(s) + b(s) (s-p_i) = 1$, where $a(s)$ is a constant, and $b(s)$ has degree less than $d(s)$.

Observe that we can write $n(s) a(s) = f(s) (s-p_i) + r$ for some complex polynomial $f(s)$ and remainder term $r \in \mathbb{C}$. 
Then, we can write $G(s)$ as the partial fraction sum
$$\trG(s) = \frac{(p_i - q_i) r}{s-p_i} + \frac{e(s)}{d(s)},$$
where $e(s) = n(s)(s-q_i) b(s) + n(s)a(s) d(s) + (p_i-q_i)f(s) d(s)$.

As a consequence, we note that if $p_i - q_i$ is small (relative to the remainder term $r$), then the Jordan block corresponding to the $p_i$ pole also has a small coefficient.
\end{lemma}

\begin{proof}
Since $s-p_i$ and $d(s)$ are coprime, there exists polynomials $a(s)$ and $b(s)$ where $a(s)d(s) + b(s)(s-p_i) = 1$, and $\operatorname{deg}(a(s)) = 0, \operatorname{deg}(b(s)) < \operatorname{deg}(d(s))$. We can then write 
\begin{align*}
    \frac{n(s)(s-q_i)}{d(s)(s-p_i)} &= \frac{n(s)(s-q_i)(a(s) d(s) + b(s) (s-p_i))}{d(s) (s-p_i)} \\
    &= \frac{n(s)(s-q_i) b(s)}{d(s)} + \frac{n(s) (s-q_i) a(s)}{s-p_i} \\
    &= \frac{n(s)(s-q_i) b(s)}{d(s)} + n(s) a(s) + \frac{(p_i - q_i)n(s)a(s)}{s-p_i}.
\end{align*}
Observe that we can write $n(s)a(s) = f(s) (s-p_i) + r$, where $r \in \mathbb{C}$. Thus, continuing from the derivations, we have
\begin{align*}
    \frac{n(s)(s-q_i)}{d(s)(s-p_i)}  = \frac{n(s)(s-q_i) b(s) + n(s)a(s) d(s) + (p_i-q_i)f(s) d(s)}{d(s)} +  \frac{(p_i - q_i)r}{s-p_i}.
\end{align*}

\end{proof}

We now consider the reverse direction, where we go from the existence of a small Jordan block corresponding to a pole to a near pole-zero cancellation at that pole.

We first recall Rouch\'e's theorem, a standard result from complex analysis \cite{stein2010complex} which will be useful for us.

\begin{theorem}[Rouch\'e's theorem]
Suppose $f$ and $g$ are holomorphic on an open set $U$ containing a circle $C$ and its interior. Suppose $\abs*{f(z)} > \abs*{g(z)}$ for all $z$ in $C$.\footnote{Technically, we only need that $\abs*{f(z) + tg(z)} > 0$ for all $t \in [0,1]$ and all $z \in C$.} Then, $f$ and $f+g$ have the same number of zeros (including multiplicity) inside the circle $C$.
\end{theorem}

\begin{lemma}
Consider a minimal SISO transfer function $\trG(s)$. Suppose $p$ is a (possibly repeated) pole of $G$, with order $n_p$. By applying partial fraction factorization, $\trG(s)$ can be (uniquely) decomposed additively as 
\begin{align*}
    \trG(s) = \sum_{j=1}^{n_p} \frac{\alpha_j}{(s-p)^j} + \frac{u(s)}{d(s)},
\end{align*}
where $u(s)$ is coprime to $d(s)$, and $d(s)$ is coprime to $(s-p)$. Then, $\min_{p' \neq p: (s-p)^ju(s) = 0} \abs*{p - p'} > 0$, and for any $\ep > 0$ such that $$\ep < \min_{p' \neq p: (s-p)^ju(s) = 0} \abs*{p - p'},$$ 
there exists $\delta(\ep,G) > 0$ such that if $\max_j \abs*{\alpha_j} < \delta(\ep,G)$, then 
\begin{align*}
    \trG(s) = \frac{\prod_{j=1}^{n_p} (x-p+ \ep_j)n(s)}{(x-p)^{n_p} d(s)}
\end{align*}
such that $\abs*{\ep_j} < \ep$ for some polynomial $n(s)$ where $n(s)$ is coprime to both $(x-p)$ and $d(s)$. 
\end{lemma}

\begin{proof}
Since 
$$ \trG(s) = \sum_{j=1}^{n_p} \frac{\alpha_j}{(s-p)^j} + \frac{u(s)}{d(s)}, $$
it follows that 
$$\trG(s) = \frac{\left(\sum_{j=1}^{n_p} \alpha_j (s-p)^{n_p-j}\right) d(s) + u(s) (s-p)^{n_p}}{(s-p)^{n_p}d(s)}.$$
Let $e(s) \coloneqq \left(\sum_{j=1}^{n_p} \alpha_j (s-p)^{n_p-j}\right) d(s)$, and let $f(s) \coloneqq u(s)(s-p)^{n_p}$. Clearly, $f(s)$ has a repeated zero with multiplicity $n_p$ at $s = p$. Consider any $\ep > 0$ such that $\ep < \min_{p' \neq p: (s-p)^j u(s) = 0} \abs*{p - p'}$; we note that $\min_{p' \neq p: (s-p)^j u(s) = 0} \abs*{p - p'}$ exists and is positive\footnote{As an aside, we note it may be infinite if $u(s)$ is a constant or if $u(s)$ only has a root at $s = p$).}. Then, it follows that $\abs*{f(s)} > 0$ for any $s \in \partial B(p, \ep) \coloneqq \{s \in \mathbb{C}: \abs*{p - s} = \ep$\}; moreover, by definition of $\ep$, the only zeros of $f(s)$ within the closed disc $B(p,\ep) \coloneqq \{s \in \mathbb{C}: \abs*{s - p} \leq \ep \}$ are at $s=p$. By continuity, there exists some $\delta(\ep,G) > 0$ such that when $\max_j \abs*{\alpha_j} < \delta(\ep,G)$, it follows that $\abs*{e(s)} < \abs*{f(s)}$ for any $s \in \partial B(p,\ep)$. In the event that holds, since $e(s)$ and $f(s)$ are holomorphic on $\mathbb{C}$, it follows from Rouch\'e's theorem that $e(s) + f(s)$ has the same number of zeros in the closed disc $B(p,\ep)$ as $f(s)$. Since $f(s)$ has precisely $n_p$ zeros within $B(p,\ep)$, our desired result follows.

\end{proof}

\section{Numerical Examples} \label{section:examples}




\subsection{Comparing balanced truncation to modal truncation} 

We here compare the performance of balanced truncation versus modal truncation. 
Consider the following plant and controller pair
$$
\mathbf{G} = \left[\begin{array}{c|c}
    \begin{bmatrix}
    -6.00 & -13.84 & -11.95 \\
    1 & 0 & 0 \\
    0 & 1 & 0
    \end{bmatrix} & \begin{bmatrix}
    1 \\ 0 \\0
    \end{bmatrix} \\
    \hline
    \begin{bmatrix}
    -1.74 & -7.63 & -8.37
    \end{bmatrix} & 0
\end{array} \right], \quad \trK \!=\! \left[\begin{array}{c|c}
    \begin{bmatrix}
    -2.1541 & -0.0104 & 0 \\
    -0.0104 & -2.1731 & 0 \\
    0 & 0 & 0.2
    \end{bmatrix} & \begin{bmatrix}
    0 \\
    1.2815 \\ 0.5
    \end{bmatrix}  \\
    \hline
    \begin{bmatrix}
    -0.8097 & -1.2368 & 0.5
    \end{bmatrix} & 0
\end{array} \right]
$$
\normalsize
It is easy to check numerically that $\trK$ internally stabilize $\mathbf{G}$. Indeed, the closed-loop poles (i.e., eigenvalues of $A_{\mathrm{cl}}$ in \cref{eq:internal-stability}) are (-0.38,-2.53, -2.15), each with multiplicity 2. Note that this controller $\mathbf{K}$ has an unstable mode $\lambda = 0.2$, thus the standard balanced truncation in \Cref{lemma:truncated_distinct_evals_assumption} is inapplicable directly. As discussed in \Cref{subsection:stable-unstable-separation}, we separate the controller as $\trK = \trK_< + \trK_\geq$, and perform reduction on the stable part $\trK_<$ (i.e. \Cref{algorithm:balanced_truncation_stable_unstable,algorithm:modal_truncation}). 

\vspace{1mm}
\begin{remark}[System generation]
To illustrate \Cref{theorem:stable_K_truncation_result} and our results in \Cref{section:controller_reduction_via_balanced_truncation,section:order_reduction_modal}, we need to consider a plant $\mathbf{G}$ and a (possibly unstable) controller $\mathbf{K}$ that is internally stabilizing.  To generate such instances, we first generate a random stable and minimal system, $\trK_<$ and another unstable part $\trK_{\geq}$. We define the augmented system $\trK = \trK_< + \trK_\geq$, and compute a stabilizing controller for $\trK$, which we call $\trG$. Finally, we treat $\trG$ as the system plant and $\trK$ as the controller. It is thus clear that $\trK$ internally stabilizes $\trG$ (duality between plant and controller). 
\end{remark}

After performing balanced truncation and modal truncation, we get two reduced-order controllers (of order 2) $\trK_{r,\mathrm{bt}}$ and $\trK_{r,\mathrm{mt}}$ respectively. Both $\trK_{r,\mathrm{bt}}$ and $\trK_{r,\mathrm{mt}}$ satisfy the bound in \cref{eq:K_Kr_pole_stability_condition_1}, thus internally stabilizes the plant, guaranteed by \Cref{theorem:stable_K_truncation_result}. Indeed, under \Cref{assumption-white}, the LQG cost of the two truncated controllers are listed in \Cref{table:compare_bal_mod_trunc_single}. 
In this example, the LQG cost of the balanced truncation is very close to the original performance, while the modal truncated controller has a slightly higher LQG cost. Note that the quantity $\norm*{\trK - \trK_r}_{\mathcal{H}_\infty}$ is significantly smaller in the case of balanced truncation. 
%
\cref{theorem:stable_K_truncation_result} states that the upper bound on $J(\trK_r) - J(\trK)$ is tighter when $\norm*{\trK - \trK_r}_{\mathcal{H}_\infty}$ is smaller; since $\norm*{\trK - \trK_{r,\mathrm{bt}}}_{\mathcal{H}_\infty} < \norm*{\trK - \trK_{r,\mathrm{mt}}}_{\mathcal{H}_\infty}$, this explains why $J(\trK_{r,\mathrm{bt}})$ is lower than $J(\trK_{r,\mathrm{mt}})$ in this case. We provide more extensive comparisons in \cref{appendix:simulation}.

\begin{table}[h]
\begin{center}
\caption{Comparison of bal. \& mod. trunc.} 
\label{table:compare_bal_mod_trunc_single}
\begin{tabular}{cccc} 
\toprule 
     & Original & Bal. Trunc. & Mod. Trunc\\
     \hline
    LQG cost & 8.0552 & 8.0552 & 8.9928 \\
    $\norm*{\trK - \trK_r}_{\mathcal{H}_\infty}$ & --- & $2.6572 \cdot 10^{-6}$ & 0.0580 \\ 
    \bottomrule
\end{tabular}
\end{center}
\end{table}

\begin{figure}[t]
    \centering
    \includegraphics[width=.45\linewidth]{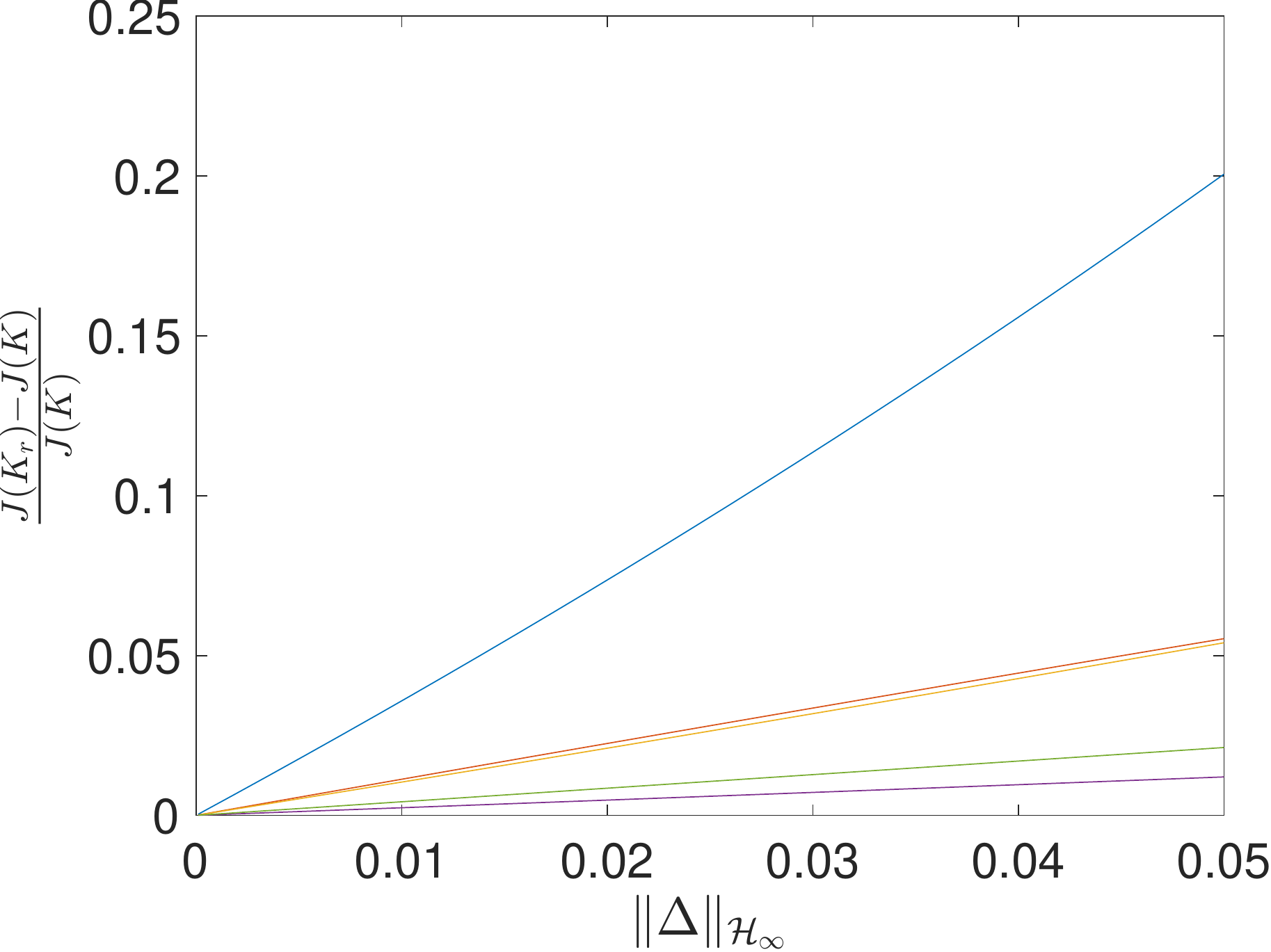}
  \caption{Relationship between LQG cost gap $J(\trK_r) - J(\trK)$ and the $\mathcal{H}_\infty$ norm of the truncated component (under modal reduction approach).}
  \label{fig:scaling_up}
\end{figure}

\textbf{Scaling effect of $\norm*{\trDelta}_{\mathcal{H}_\infty}$.} We next study how the performance gap behaves as a function of the size of the truncated component. For this analysis, we (randomly) generate five stable and minimal SISO systems of order 4, $\trK_r$, and augment the system by adding a stable mode $\trDelta$, where
\begin{align*}
        \trK_r \!= \! \left[\begin{array}{c|c}
        \begin{bmatrix}
        1.5 & -1 & -0.21 \\
        3.00 & -0.43 & -1.00 \\
        2.00 & -0.07 & -5.00
        \end{bmatrix} & \begin{bmatrix}
        0.18 \\
        0.97 \\
        1.2
        \end{bmatrix}  \\
        \hline
        \\[-8pt]
        \begin{bmatrix}
        1.00 & 2.00 & 3.00
        \end{bmatrix} & 0
    \end{array} \right],  \ \trDelta \! = \! \left[\begin{array}{c|c}
    -1 & \sqrt{\eps} \\
    \hline
    \sqrt{\eps} & 0
\end{array} \right].
\end{align*}
\normalsize
We denote the augmented system as $\trK := \trK_r + \trDelta$. We then generate a stabilizing controller, $\trG$, which stabilizes $\trK$.  Viewing $\trG$ as the system plant, we then compare the LQG cost of $\trK$ and $\trK_r$ on the system $\trG$, as the $\mathcal{H}_\infty$ norm of the truncated component, i.e. $\norm*{\trDelta}_{\mathcal{H}_\infty}$, varies (to be precise, we plot 30 $\trDelta$'s, each corresponding to a different $\eps$, where we let $\eps$ range (equally spaced) between 0.0001 and 0.05). As we can see in \cref{fig:scaling_up}, for the five set of controllers, there is close to a linear relationship (with different slopes) between the LQG cost gap ratio $\frac{J(\trK_r) - J(\trK)}{J(\trK)}$ and the $\mathcal{H}_\infty$ norm of the truncated component, i.e. $\norm*{\trDelta}_{\mathcal{H}_\infty}$; this is consistent with the upper bound on $J(\trK_r)$ in \cref{theorem:modal_truncation_SISO}. 


\subsection{Truncating unstable mode(s)} 


We here consider an example to illustrate \Cref{theorem:modal_truncation_SISO}. Consider the following plant $\trG$
\begin{align*}
    \trG = \left[\begin{array}{c|c}
        \begin{bmatrix}
        -5.86 & -9.50 & 0.56 \\
        1 & 0 & 0 \\
        0 & 1 & 0
        \end{bmatrix} & \begin{bmatrix}
        1 \\
        0 \\
        0
        \end{bmatrix} \\
        \hline
        \\[-10pt]
        \begin{bmatrix}
        -7.18 & -25.61 & -8.41
        \end{bmatrix} & 0
    \end{array}\right].
\end{align*}
\normalsize
It can be verified that the controller
$
    \trK = \left[\begin{array}{c|c}
    A_\mK & B_\mK \\
    \hline 
    C_\mK & 0
    \end{array} \right], 
$
with 
%
$$A_\mK\! =\! \begin{bmatrix}
        1.37 & 0 & 0  \\
        0 & -0.37 & 0 \\
        0 & 0 & 0.34
    \end{bmatrix}, \ B_\mK \!=\! \begin{bmatrix}
    0.19 \\
    0.04 \\
    0.04
    \end{bmatrix}, \ C_\mK \!=\! \begin{bmatrix}
    3.79 \\ 4.14 \\ -1.57
    \end{bmatrix}^\tr. $$
\normalsize
internally stabilizes $\trG$. 
Applying \Cref{algorithm:modal_truncation}, we obtain an order 2 controller $\trK_r$, which removes the last (unstable) mode of $A_\mK$, leading to 
$
    \trK_r = \left[\begin{array}{c|c}
    A_{\mK,r} & B_{\mK,r} \\
    \hline 
    C_{\mK,r} & 0
    \end{array} \right], 
$
 with 
$$\quad A_{\mK,r} \!=\! \begin{bmatrix}
        1.37 & 0   \\
        0 & -0.37 
    \end{bmatrix}, \; B_{\mK,r}\! =\! \begin{bmatrix}
    0.19 \\
    0.04
    \end{bmatrix}, \; C_{\mK,r} \!=\! \begin{bmatrix}
    3.79 \\ 4.14
    \end{bmatrix}^\tr,$$
    \normalsize
The truncated component $\trDelta$ is unstable and takes the form
\begin{align*}
    \trDelta = \left[\begin{array}{c|c}
    0.34 & 0.04 \\
    \hline 
    -1.57 & 0
    \end{array} \right],
\end{align*}
\normalsize
which satisfies the bound \cref{eq:K_Kr_pole_stability_condition_unstable}. Thus, \Cref{theorem:modal_truncation_SISO} guarantees that the reduced-order controller $\trK_r $ still internally stabilizes the plant. Indeed, numerical computation shows that the LQG cost of the original controller, $J(\trK)$, is 343.2, while the LQG cost of the truncated controller, $J(\trK_r)$, is 58.2. In this case , modal truncation not only yields a stabilizing lower-order controller, but also a cost of lower LQG cost.\footnote{For this instance, the theoretical upper bound posited in \Cref{theorem:modal_truncation_SISO} is significantly larger than the original cost.}

\section{Conclusion} \label{section:conclusion}


We have presented on controller reduction for general non observer-based controllers using balanced truncation and modal truncation. For SISO systems, we demonstrate how LQG control may be performed even when there are no stable components in the controller. We hope that our work will be useful not only for policy optimization in LQG control but also for the controller reduction community. Two interesting future directions are 1) extending truncation of unstable modes to MIMO systems and 2) applying the results to escape saddle points in the LQG policy optimization \cite{zheng2022escaping}.

\appendix




\section{Background on  Balanced realization}
\label{subsec:balanced_realization}
The following result, which states that any stable transfer function has a balanced realization, is well-known from classical control theory.

\begin{proposition}[cf.  {\cite[Section~26.2]{dahleh2004lectures}}]
\label{prop:balanced_realization}
Consider a stable transfer function $\trG$ with a minimal state-space realization 
$$\trG = \left[\begin{array}{c|c}
    A & B \\
     \hline
    C &  D
\end{array}\right],$$
i.e. the transfer function can be implemented in state-space as the linear system
\begin{align*}
    &\dot{x} = Ax + Bu, \\
    &y = Cx + Du,
\end{align*}
for some $A \in \bbR^{n \times n}, B \in \bbR^{n \times p}, C \in \bbR^{m \times n}$ and $D \in \bbR^{m \times p}$,
where $A \in \bbR^{n \times n}$ is a stable matrix, and $(A,B)$ is controllable and $(A,C)$ is observable. Recall that the corresponding controllability and observability gramians $W_c$ and $W_o$ are defined as the positive definite\footnote{The positive definite nature of the solutions follows from the definition of minimality.} solutions to the Lyapunov equations
\begin{subequations}
\begin{align}
    &AW_c + W_cA^\tr + BB^\tr = 0, \label{eq:Lyapunov_Wc} \\
    &A^\tr W_o + W_o A + C^\tr C = 0 \label{eq:Lyapunov_Wo}
\end{align}
\end{subequations}
respectively.
Then, there exists a balanced state-space realization of $\trG$ in the form
$$\trG = \left[\begin{array}{c|c}
   \tilde{A}  & \tilde{B}  \\ 
     \hline 
     \\[-10pt]
    \tilde{C} &  \tilde{D}
\end{array} \right],$$
where 
$$\tilde{A} = TAT^{-1}, \ \tilde{B} = TB, \ \tilde{C} = CT^{-1}, \ \ \tilde{D} = D$$ 
for some invertible $T \in \bbR^{n \times n}$, such that the corresponding controllability and observability gramians $\tilde{W}_c = TW_cT^{\tr}$ and $\tilde{W}_o = T^{-\tr} W_o T^{-1}$ satisfy the relation
$$\tilde{W}_c = \tilde{W}_o = \Lambda$$
for some positive diagonal matrix $\Lambda \in \bbR^{n \times n}$.
\end{proposition}

\begin{proof}
First, for any invertible transformation $T \in \bbR^{n \times n}$, we note that if we define the coordinate $\tilde{x} = Tx$, then the transfer function $\trG$ can be implemented in state-space as the linear system
\begin{align*}
    &\tilde{x} = Tx = TAT^{-1}\tilde{x} + TBu, \\
    &y = CT^{-1} \tilde{x} + Du.
\end{align*}
Henceforth, for notational convenience we denote $\tilde{A} = TAT^{-1}, \tilde{B} = TB, \tilde{C} = CT^{-1}, \tilde{D} = D$. 
The corresponding controllability and observability gramians $\tilde{W}_c$ and $\tilde{W}_o$ are the solutions to the Lyapunov equations
\begin{align*}
    &TAT^{-1} \tilde{W}_c + \tilde{W}_c(TAT^{-1})^\tr + TB(TB)^\tr = 0, \\
    &(TAT^{-1})^\tr \tilde{W}_o + \tilde{W}_o TAT^{-1} + (CT^{-1})^\tr CT^{-1} = 0
\end{align*}
By multiplying \cref{eq:Lyapunov_Wc} on the left by $T$ and on the right by $T^\tr$, and similarly multiplying \cref{eq:Lyapunov_Wo} on the left by $T^{-1}$ and on the right by $T^{-\tr}$, it follows that $$\tilde{W}_c = TW_cT^\tr, \qquad \tilde{W}_o = T^{-\tr} W_o T^{-1},$$
where both $\tilde{W}_c$ and $\tilde{W}_o$ are symmetric, and also positive definite since $W_c$ and $W_o$ are positive definite. Our goal is to find a $T$ such that $\tilde{W}_c = \tilde{W}_o = \Lambda$ for some positive definite diagonal matrix $\Lambda$, i.e. $\tilde{W}$ and $\tilde{W}_o$ are equal to each other and are a diagonal matrix with positive entries along the diagonal. Moreover, we will see that $\Lambda^2$ will share the same eigenvalues as the matrix product $W_c W_o$. 

Suppose such a $T$ exists. That implies then that $$T W_c W_o T^{-1} = (T W_cT^\tr)(T^{-\tr} W_o T^{-1}) = \Lambda \Lambda =  \Lambda^2.$$ 
Now, note that $W_c$ can be written in the form $W_c = QQ^\tr$ for some invertible $Q$, for instance by taking its Cholesky decomposition. Then, we have
$$\Lambda^2  = TW_c W_o T^{-1} = TQQ^\tr W_o T^{-1} = (TQ)Q^\tr W_o Q (TQ)^{-1}.$$
This implies that $Q^\tr W_o Q$ is similar to $\Lambda^2$, and hence there exists an orthonormal matrix $U$ such that $Q^\tr W_o Q = U \Lambda^2 U^\tr$. Thus, we have
$$\Lambda^2 = (TQ)Q^\tr W_o Q (TQ)^{-1} = (TQ)U\Lambda^2 U^\tr (TQ)^{-1} = TQ U \Lambda^{-1/2} \Lambda^2 \Lambda^{1/2} U^\tr Q^{-1} T^{-1}.$$
Set $TQU\Lambda^{-1/2} = I$, such that $T = \Lambda^{1/2} U^\tr Q^{-1}$. Then, 
\begin{align*}
    &TW_cT^\tr = (\Lambda^{1/2} U^\tr Q^{-1})QQ^\tr (\Lambda^{1/2} U^\tr Q^{-1})^\tr = \Lambda^{1/2} U^\tr U \Lambda^{1/2} = \Lambda, \\
    &T^{-\tr} W_o T^{-1} = (\Lambda^{1/2} U^\tr Q^{-1})^{-\tr} W_o (\Lambda^{1/2} U^\tr Q^{-1})^{-1} =  \Lambda^{-1/2}U Q^\tr W_o Q U^\tr \Lambda^{1/2} = \Lambda^{-1/2} \Lambda^2 \Lambda^{1/2} = \Lambda.
\end{align*}
Hence, the desired result follows by applying the coordinate transformation $\tilde{x} = Tx,$ where $T = \Lambda^{1/2} U^\tr Q^{-1}$. Note that we have not explicitly specified what the positive definite diagonal matrix $\Lambda$ should be. To this end, we note that since $TW_cW_oT^{-1}$ is a similarity transformation of $W_c W_o$, $\Lambda^2 = TW_cW_oT^{-1}$ must share the same eigenvalues as $W_c W_o$. Hence one possible choice for $\Lambda$ is choosing $\Lambda \coloneqq  \diag\left(\{\lambda_i(W_cW_o)\}_{i=1}^n\right)$ for each $i \in [n]$. This completes our proof.


\end{proof}

\section{Useful norm identities and inequalities}
\label{appendix:background_norm_bounds}

For the reader's reference, we collect several useful well-known norm identities and inequalities that recur frequently in our proofs. The reader may refer to \citep[Chapter 4]{zhou1996robust} for more details.

We begin with the following norm triangular inequalities
\begin{align*}
    &\norm*{\trG_1 + \trG_2}_{\mathcal{L}_2} \leq \norm*{\trG_1}_{\mathcal{L}_2} + \norm*{\trG_2}_{\mathcal{L}_2}, \quad \quad \forall \trG_1, \trG_2, \\
    &\norm*{\trG_1 + \trG_2}_{\mathcal{L}_\infty} \leq \norm*{\trG_1}_{\mathcal{L}_\infty} + \norm*{\trG_2}_{\mathcal{L}_\infty}, \quad \forall \trG_1, \trG_2.
\end{align*}
As a corollary, 
\begin{align*}
    &\norm*{\trG_1 + \trG_2}_{\mathcal{H}_2} \leq \norm*{\trG_1}_{\mathcal{H}_2} + \norm*{\trG_2}_{\mathcal{H}_2}, \quad \quad \forall \trG_1, \trG_2 \in \mathcal{R} \mathcal{H}_2, \\
    &\norm*{\trG_1 + \trG_2}_{\mathcal{H}_\infty} \leq \norm*{\trG_1}_{\mathcal{H}_\infty} + \norm*{\trG_2}_{\mathcal{H}_\infty}, \quad \forall \trG_1, \trG_2 \in \mathcal{R} \mathcal{H}_\infty
\end{align*}

The following result, bounding the $\mathcal{H}_\infty$ norm of the products of two transfer functions using a sub-multiplicative property, is used throughout our proofs.

\begin{lemma}
\label{lemma:G1_G_2_norm_ineq_H_L_mixed_inf}
Let $\trG_1 \in \mathcal{R}\mathcal{H}_\infty$. Suppose that $\trG_2$ has finite $\mathcal{L}_\infty$ norm, and $\trG_1$ and $\trG_2$ have compatible dimensions. If $\trG_1 \trG_2$ is stable, then
\begin{align*}
    \norm*{\trG_1 \trG_2}_{\mathcal{H}_\infty} \leq \norm*{\trG_1}_{\mathcal{H}_\infty} \norm*{\trG_2}_{\mathcal{L}_\infty}. 
\end{align*}
As a special case, when $\trG_2 \in \mathcal{RH}_\infty$ as well, we have 
\begin{align*}
    \norm*{\trG_1 \trG_2}_{\mathcal{H}_\infty} \leq \norm*{\trG_1}_{\mathcal{H}_\infty} \norm*{\trG_2}_{\mathcal{H}_\infty}.
\end{align*}
\end{lemma}
\begin{proof}
Suppose $\trG_1 \trG_2$ is stable. Then, by definition, its $\mathcal{H}_\infty$ norm is the same as its $\mathcal{L}_\infty$ norm. Hence, 
\begin{align*}
    \norm*{\trG_1 \trG_2}_{\mathcal{H}_\infty} = \norm*{\trG_1 \trG_2}_{\mathcal{L}_\infty} &= \sup_{w \in \bbR} \sigma_{\max}(\trG_1(jw)\trG_2(jw)) \\
    &\leq \sup_{w \in \bbR}  \sigma_{\max}(\trG_1(jw)) \sup_{w \in \bbR}  \sigma_{\max}(\trG_2(jw)) \\
    &\leq \norm*{\trG_1}_{\mathcal{H}_\infty} \norm*{\trG_2}_{\mathcal{L}_\infty}.
\end{align*}
The special case follows from the definition of the $\mathcal{H}_\infty$ norm.
\end{proof}

The next result, which bounds the $\mathcal{H}_2$ norm of the products of two transfer functions, is also quite useful. 

\begin{lemma}
\label{lemma:G1_G_2_norm_ineq_H_L_mixed_2}
Let $\trG_1 \in \mathcal{R}\mathcal{H}_\infty$. Suppose that $\trG_2$ has finite $\mathcal{L}_2$ norm, and $\trG_1$ and $\trG_2$ have compatible dimensions. If $\trG_1 \trG_2$ is stable, then
\begin{align}
    \label{eq:G1_G_2_norm_ineq_H_L_mixed_2_G1_finite_Hinf}
    \norm*{\trG_1 \trG_2}_{\mathcal{H}_2} \leq \norm*{\trG_1}_{\mathcal{H}_\infty} \norm*{\trG_2}_{\mathcal{L}_2}. 
\end{align}
In addition, when $\trG_1 \in \mathcal{R}\mathcal{H}_2$, and suppose that $\trG_2$ has finite $\mathcal{L}_\infty$ norm, where $\trG_1$ and $\trG_2$ have compatible dimensions. If $\trG_1 \trG_2$ is stable, then
\begin{align}
    \label{eq:G1_G_2_norm_ineq_H_L_mixed_2_G1_finite_H2}
    \norm*{\trG_1 \trG_2}_{\mathcal{H}_2} \leq \norm*{\trG_1}_{\mathcal{H}_2} \norm*{\trG_2}_{\mathcal{L}_\infty}. 
\end{align}
As a corollary, when $\trG_1, \trG_2$ are both in $\mathcal{R}\mathcal{H}_2$, we have 
\begin{align*}
    \norm*{\trG_1 \trG_2}_{\mathcal{H}_2} \leq \norm*{\trG_1}_{\mathcal{H}_\infty} \norm*{\trG_2}_{\mathcal{H}_2}, \quad     \norm*{\trG_1 \trG_2}_{\mathcal{H}_2} \leq \norm*{\trG_1}_{\mathcal{H}_2} \norm*{\trG_2}_{\mathcal{H}_\infty}.
\end{align*}
\end{lemma}
\begin{proof}
Suppose $\trG_1 \trG_2$ is stable. Then, by definition, its $\mathcal{H}_2$ norm is the same as its $\mathcal{L}_2$ norm. Suppose that $\trG_1 \in \mathcal{R} \mathcal{H}_\infty$ and that $\trG_2$ has finite $\mathcal{L}_2$ norm. Then, 
\begin{align*}
    \norm*{\trG_1 \trG_2}_{\mathcal{H}_2}^2 = \norm*{\trG_1 \trG_2}_{\mathcal{L}_2}^2 &= \int_{0}^\infty \trace \left((\trG_1(jw) \trG_2(jw))^* (\trG_1(jw)\trG_2(jw)) \right) dw \\
    &= \int_{0}^\infty \trace \left(\trG_2(jw)^* \trG_1(jw)^* \trG_1(jw)\trG_2(jw) \right) dw \\
    &\leq \int_{0}^\infty \sigma_{\max}(\trG_1^*(jw) \trG_1(jw)) \trace \left(\trG_2(jw)^*\trG_2(jw) \right) dw \\
    &\leq \norm*{\trG_1}_{\mathcal{L}_\infty}^2 \norm*{\trG_2}_{\mathcal{L}_2}^2 = \norm*{\trG_1}_{\mathcal{H}_\infty}^2 \norm*{\trG_2}_{\mathcal{L}_2}^2
\end{align*}
Using analogous calculations we may derive the other inequalities in our result as well. Our corollary follows from the definition of the $\mathcal{H}_\infty$ and $\mathcal{H}_2$ norm.
\end{proof}

We end with the following inequality which follows from the small-gain theorem \citep[Theorem 9.1]{zhou1996robust}. 

\begin{lemma} \label{lemma:small-gain-type-inequality}
Suppose $\trG_1 \trG_2 \in \mathcal{R} \mathcal{H}_\infty$ and $\norm*{\trG_1}_{\mathcal{H}_\infty} \norm*{\trG_2}_{\mathcal{L}_\infty} < 1$. It follows that 
\begin{align*}
    \norm*{(I - \trG_1 \trG_2)^{-1}}_{\mathcal{H}_\infty} \leq \frac{1}{1 - \norm*{\trG_1}_{\mathcal{H}_\infty} \norm*{\trG_2}_{\mathcal{L}_\infty}}.
\end{align*}
As a corollary, if $\trG_2 \in \mathcal{R}\mathcal{H}_\infty$ as well, then 
\begin{align*}
    \norm*{(I - \trG_1 \trG_2)^{-1}}_{\mathcal{H}_\infty} \leq \frac{1}{1 - \norm*{\trG_1}_{\mathcal{H}_\infty} \norm*{\trG_2}_{\mathcal{H}_\infty}}.
\end{align*}
\end{lemma}

\section{Technical proofs}

\subsection{Proof of \cref{lemma:LQG-cost}}
\label{appendix:proof_of_LQG-cost}

\begin{proof}[Proof of \cref{lemma:LQG-cost}]
It follows from the definition of the LQG cost in \cref{eq:LQG-cost} that we have
$$J(\trK) = \lim_{t \to \infty} \bbE\left[z(t)^\top z(t) \right], \quad z(t) \coloneqq \begin{bmatrix}
\tilde{y}(t) \\
u(t)
\end{bmatrix}.$$
Denote $\tilde{w}(t) \coloneqq \begin{bmatrix}
w(t) &
v(t)
\end{bmatrix}^\top.$ Denote
\begin{align*}
    \mathbf{T} \coloneqq \begin{bmatrix}
    (I - \mathbf{G}\mathbf{K})^{-1} \mathbf{G} & (I - \mathbf{G}\mathbf{K})^{-1} \mathbf{G} \mathbf{K} \\
    \mathbf{K}(I - \mathbf{G}\mathbf{K})^{-1} \mathbf{G} &  \mathbf{K} (I - \mathbf{G}\mathbf{K})^{-1}
    \end{bmatrix}.
\end{align*}
Then, from \cref{eq:y_tilde_u_w_v_freq_rship}, we have that $\mathbf{z} = \mathbf{T} \mathbf{\tilde{w}}$. Equivalently, there exists some (proper) state-space realization $(A_{\mathbf{T}},B_{\mathbf{T}}, C_{\mathbf{T}},0)$ of $\mathbf{T}$, such that we have the dynamical representation
\begin{align*}
    &\tilde{\xi}(t) = A_{\mathbf{T}} \tilde{\xi}(t) + B_{\mathbf{T}}\tilde{w}(t) \\
    &z(t) = C_{\mathbf{T}} \tilde{\xi}(t).
\end{align*}
This implies in particular that
\begin{align*}
    z(t) = C_{\mathbf{T}} \int_{\tau=0}^t e^{A_{\mathbf{T}}(t-\tau)}B_{\mathbf{T}} \tilde{w}(\tau) d\tau.
\end{align*}
Thus, assuming that $w$ and $v$ are white noise terms (so $\tilde{w}$ is also a white noise term), we have
\begin{align*}
    J(\trK) &= \lim_{t \to \infty} \bbE\left[z(t)^\top z(t) \right] \\
    &= \lim_{t \to \infty} \bbE\left[\left(C_{\mathbf{T}} \int_{\tau=0}^t e^{A_{\mathbf{T}}(t-\tau)}B_{\mathbf{T}} \tilde{w}(\tau) d\tau\right)^\top \left(C_{\mathbf{T}} \int_{\tau=0}^t e^{A_{\mathbf{T}}(t-\tau)}B_{\mathbf{T}} \tilde{w}(\tau) d\tau\right) \right] \\
    &= \bbE\left[\trace \left( \int_{0}^\infty \left(C_{\mathbf{T}}e^{A_{\mathbf{T}}t} B_{\mathbf{T}} \right)^\top \left(C_{\mathbf{T}}e^{A_{\mathbf{T}}t} B_{\mathbf{T}} \right)dt \right) \right] \\
    &= \frac{1}{2\pi} \int_{-\infty}^\infty \trace \left(\mathbf{T}^*(j\omega) \mathbf{T}(j\omega) \right) d\omega \\
    &= \norm*{\mathbf{T}}_{\mathcal{H}_2}^2
\end{align*}
where the second-to-last equality holds by Parseval's theorem, and the final equality follows from definition of the $\mathcal{H}_2$ norm. This completes our proof.
\end{proof}

\subsection{Proof of \cref{lemma:K_perturbed_still_stable_conditions_old}} \label{appendix:lemma_3_proof}
Although \cref{lemma:K_perturbed_still_stable_conditions_old} is widely used, we cannot find a proof easily in the literature. 
We provide here a proof of \cref{lemma:K_perturbed_still_stable_conditions_old}, which relies on the MIMO version of the Nyquist stability theorem \cite[Theorem 5.8]{zhou1996robust} and the discussion in \cite[Section IV]{doyle1981multivariable}. 

Throughout our proof below, given a transfer function $\mathbf{F}$, we use $n_z(\mathbf{F})$ to denote the number of zeros of $\mathbf{F}$ in the open right-half complex plane, and $n_p(\mathbf{F})$ to denote the number of poles of $\mathbf{F}$ in the open right-half complex plane.

\begin{proof}[Proof of \cref{lemma:K_perturbed_still_stable_conditions_old}]
Recall that $\trDelta := \trK - \trK_r$. Without loss of generality, it suffices for us to show the result in the case that $\norm*{(I - \trG \trK)^{-1} \trG \trDelta}_{\mathcal{L}_\infty} < 1$; the other case, when $\norm*{\trDelta \trG(I - \trK \trG)^{-1}}_{\mathcal{L}_\infty} < 1$, follows by symmetry of $\trK$ and $\trG$ in the closed-loop transfer matrix.

By \citep[Theorem 5.7]{zhou1996robust}, to prove that $\trK_r$ internally stabilizes $\trG$, it suffices to show two items. 
\begin{enumerate}
    \item $n_p(\trG) + n_p(\trK_r) = n_p(\trG \trK_r)$, i.e., no unstable pole-zero cancellation between $\trG$ and $\trK_r$. 
    \item $(I - \trG \trK_r)^{-1}$ is stable, i.e. $\det((I - \trG  \trK)^{-1})$ has all its poles in the open left-half complex plane.
\end{enumerate}

We will first show that $W(\det(I - \trG \trK_r)) = W(\det(I - \trG \trK))$, where $W(\cdot)$ denotes the winding number (around the origin) of a trajectory in $\mathbb{C}$. 
By homotopy invariance of the winding number, if $\det((I - \trG \trK + \ep \trG(\trK - \trK_r))(s)) \neq 0$ for all $s \in D$ and all $0 \leq \ep \leq 1$, where $D$ denotes the Nyquist contour in the right-half-plane\footnote{This is a contour comprising a path travelling up the $iw$ axis, from $0-i\infty$ to $0 + i \infty$, and another semicircular arc, with radius $r \to \infty$, that travels clockwise from $0 + i\infty$ to $0 - i\infty$.}, it follows that $W(\det(I - \trG \trK_r)) = W(\det(I - \trG \trK))$. We now prove that $\forall 0 \leq \epsilon \leq 1$, we have $\det((I - \trG \trK + \ep \trG(\trK - \trK_r))(s)) \neq 0$ for all $s \in D$. 

\begin{itemize}
    \item For the semicircular arc at infinity of $D$, it is not hard to see  that $\det\left((I - \trG \trK + \ep \trG(\trK - \trK_r))\right) \not\equiv 0$ for any $0 \leq \ep \leq 1$. This is because the determinant of any transfer function is a complex polynomial and thus has finitely many roots. Therefore it must be non-vanishing as we approach infinity. 
    \item


We here show that $\det\left((I - \trG \trK + \ep\trG(\trK - \trK_r))(i\omega)\right) \neq 0$ for any $\omega \in \bbR$. This is equivalent to showing
$$\underline{\sigma}(I - \trG \trK + \ep\trG(\trK - \trK_r)) > 0,$$
where  for any transfer function $\mathbf{F}$, we denote $\underline{\sigma}(\mathbf{F}) := \min_{\omega \in \mathbb{R}}\, \sqrt{\lambda_{\min}(\mathbf{F}(-i\omega)^\tr \mathbf{F}(i\omega))},$ i.e., the minimum singular value of $\mathbf{F}(s)$ over the imaginary axis.

Since $\norm*{ (I - \trG \trK)^{-1} \trG (\trK- \trK_r) }_{\mathcal{L}_\infty} < 1$, it follows that for any $ 0 \leq \ep  \leq 1$, 
\begin{align*}
    \underline{\sigma}(I - \trG \trK + \ep \trG(\trK - \trK_r)) &= \underline{\sigma}((I - \trG \trK)(I + \ep (I - \trG \trK)^{-1} \trG (\trK - \trK_r))) \\
    &\labelrel\geq{eq:bound1} \underline{\sigma}((I - \trG \trK)) (1 - \norm*{(I - \trG \trK)^{-1} \trG (\trK - \trK_r)}_{\mathcal{L}_\infty}) \\
    &\labelrel>{eq:use_K_stabilizes_G_and_l_infty_bound} 0.
\end{align*}
To derive \eqref{eq:bound1}, we used the fact that $\sigma_{\min}(AB) \geq \sigma_{\min}(A)\sigma_{\min}(B)$ and $\sigma_{\min}(I + A) \geq 1 - \sigma_{\max}(A)$ for compatible matrices. 
To derive \eqref{eq:use_K_stabilizes_G_and_l_infty_bound}, we used the fact that 
$1 - \norm*{(I - \trG \trK)^{-1} \trG (\trK - \trK_r)}_{\mathcal{L}_\infty} > 0,$  
and the fact that $\trK$ internally stabilizes $\trG$ (which implies that $\det(I - \trG \trK)$ has all its zeros in the open left-half plane), indicating that $\underline{\sigma}((I - \trG \trK)) > 0$ over the imaginary axis.
\end{itemize}

We thus conclude that 
\begin{align}
\label{eq:winding_numbers_equal}
    W(\det(I - \trG \trK)) = W(\det(I - \trG \trK_r))
\end{align}
By the multivariate Nyquist criterion, we know that for any square transfer function $\mathbf{F}$,
$$W(\det(I - \mathbf{F})) = n_z(I - \mathbf{F}) - n_p(\mathbf{F}).$$


Since $\trK$ internally stabilizes $\trG$, we know that $n_z(I - \trG \trK) = 0$. Thus, 
we have $W(\det(I - \trG \trK)) = - n_p(\trG \trK).$
Combining this with \cref{eq:winding_numbers_equal}, we then have that
\begin{align}
\label{eq:winding_number_equal_n_p}
    W(\det(I - \trG \trK_r)) 
    = - n_p(\trG \trK).
\end{align}
Since $n_z(I - \trG \trK_r) \geq 0$, we also know that  
\begin{align}
\label{eq:nyquist_GKr}
W(\det(I - \trG \trK_r)) = n_z(I - \trG \trK_r) - n_p(\trG \trK_r) \geq - n_p(\trG \trK_r)
\end{align}
Combining \cref{eq:winding_number_equal_n_p} and \cref{eq:nyquist_GKr} leads to
\begin{align}
\label{eq:n_p(GK)_vs_n_p(GKr)}
    n_p(\trG \trK) \leq n_p(\trG \trK_r). 
\end{align}
Since $\trK$ internally stabilizes $\trG$, $n_p(\trG \trK) = n_p(\trG) + n_p(\trK)$. Meanwhile, by definition we have 
$$n_p(\trG \trK_r) \leq n_p(\trG) + n_p(\trK_r) = n_p(\trG) + n_p(\trK) = n_p(\trG \trK),$$
where we used the assumption that $n_p(\trK) = n_p(\trK_r)$. Continuing from \cref{eq:n_p(GK)_vs_n_p(GKr)}, we then conclude that
$$n_p(\trG \trK_r) = n_p(\trG \trK) = n_p(\trG) + n_p(\trK) =  n_p(\trG) + n_p(\trK_r).$$
Thus, there is no unstable pole-zero cancellation between $\trG$ and $\trK_r$. 

Meanwhile, we also have that
\begin{align*}
    n_z(I - \trG \trK_r) &= W(\det(I - \trG \trK_r)) + n_p(\trG \trK_r) \\
    &= W(\det(I - \trG \trK)) + n_p(\trG \trK) = 0.
\end{align*}
So, $(I - \trG \trK_r)^{-1}$ has no pole in the closed RHP. Since we already showed earlier that for any $0 \leq \ep \leq 1,$
$$\det(I - \trG \trK + \ep \trG(\trK - \trK_r)) \neq 0,$$
since $(I - \trG \trK)^{-1}$ has no pole along the imaginary axis, it follows that $(I - \trG \trK_r)^{-1}$ has no pole on the imaginary axis as well. Thus, $(I - \trG \trK_r)^{-1}$ is stable. This concludes our proof.
\end{proof}

\subsection{Proof of the LQG bound in \cref{eq:LQG-unstable-SISO}} \label{appendix:proof-LQG-bound}

We here prove the results on the change in LQG performance in \cref{eq:LQG-unstable-SISO}. We have proved that the conditions in \cref{theorem:truncated_K_still_stable_conditions} are all satisfied and thus $\trK_r$ internally stabilizes $\trG$.

From \Cref{lemma:LQG-cost}, we know that 
\begin{equation*}
    J(\mathbf{K}_r) \!=\! \norm*{(I \!-\! \mathbf{G}\mathbf{K}_r)^{-1}\mathbf{G}}_{\mathcal{H}_2}^2 + \norm*{(I \!-\! \mathbf{G}\mathbf{K}_r)^{-1}\mathbf{G} \trK_r}_{\mathcal{H}_2}^2 + \norm*{\trK_r(I \!-\! \mathbf{G}\mathbf{K}_r)^{-1}\mathbf{G}}_{\mathcal{H}_2}^2 + \norm*{\trK_r(I \!-\! \mathbf{G}\mathbf{K}_r)^{-1}}_{\mathcal{H}_2}^2. 
\end{equation*}
We proceed to upper bound the $\mathcal{H}_2$ norm of each sub-block 
in terms of norms of the sub-blocks corresponding to $\trK$. 


    Observe that by \cref{eq:I-GK_r_inv_G_expression}, we have that 
    \begin{align}
        \norm*{(I - \mathbf{G}\mathbf{K}_r)^{-1}\mathbf{G}}_{\mathcal{H}_2} &= \norm*{(I - \mathbf{X} \mathbf{\Delta})^{-1} (I - \mathbf{G}\mathbf{K})^{-1}\mathbf{G}}_{\mathcal{H}_2} \nonumber \\
        &\leq \norm*{(1 - \trX \trDelta)^{-1}}_{\mathcal{H}_{\infty}} \left(\norm*{(I - \mathbf{G}\mathbf{K})^{-1}\mathbf{G}}_{\mathcal{H}_2} \right). \label{eq:first_K_r_sub-block_H2_diff-new}
    \end{align}
    Next, observe that
    \begin{align*}
       &\norm*{(I - \mathbf{G}\mathbf{K}_r)^{-1}\mathbf{G} \mathbf{K}_r}_{\mathcal{H}_2} = \norm*{(I - \mathbf{G}\mathbf{K}_r)^{-1}\mathbf{G} (\mathbf{K} + \mathbf{\Delta})}_{\mathcal{H}_2} \nonumber \\
       \leq \,  &\norm*{(I - \mathbf{G}\mathbf{K}_r)^{-1}\mathbf{G} \mathbf{K}}_{\mathcal{H}_2} + \norm*{(I - \mathbf{G}\mathbf{K}_r)^{-1}\mathbf{G} \mathbf{\Delta}}_{\mathcal{H}_2} \nonumber \\
       \leq \, &\norm*{(1 - \trX \trDelta)^{-1}}_{\mathcal{H}_{\infty}} \left(\norm*{(I - \mathbf{G}\mathbf{K})^{-1}\mathbf{G} \mathbf{K}}_{\mathcal{H}_2} \right) + \norm*{(1 - \trX \trDelta)^{-1}}_{\mathcal{H}_{\infty}}  \norm*{\mathbf{\Delta}}_{L_{\infty}} \left(\norm*{(I - \mathbf{G}\mathbf{K})^{-1}\mathbf{G}}_{\mathcal{H}_2} \right), 
    \end{align*}
    where we used \cref{eq:first_K_r_sub-block_H2_diff} and an analogous calculation to \cref{eq:I-GK_r_inv_G_expression} (but keeping the $\mathbf{K}$ on the right) to derive the last inequality. We also used \cref{eq:G1_G_2_norm_ineq_H_L_mixed_2_G1_finite_H2} in \cref{appendix:background_norm_bounds} to bound the $\mathcal{H}_2$ norm of the term $$\norm*{(I - \trG \trK_r)^{-1} \trG \trDelta}_{\mathcal{H}_2} \leq \norm*{\trDelta}_{\mathcal{L}_\infty} \norm*{(I - \trG \trK_r)^{-1} \trG}_{\mathcal{H}_2}.$$
    
    We next consider the term $\mathbf{K}_r (I - \mathbf{G} \mathbf{K}_r)^{-1}$. From \cref{eq:Kr(I-GKr)_inv_expression} in the proof of \cref{theorem:truncated_K_still_stable_conditions}, we have 
    \begin{align*}
        \mathbf{K}_r (I - \mathbf{G} \mathbf{K}_r)^{-1} = \trK(I - \mathbf{G}\trK)^{-1} + \trK \trX  (I - \trDelta \mathbf{X})^{-1} \mathbf{\Delta}  (I - \mathbf{G} \mathbf{K})^{-1} + (I - \trDelta\mathbf{X} )^{-1}\trDelta (I - \mathbf{G}\trK)^{-1}.
    \end{align*}
  %
  %
    Thus, we have that
    \begin{align*}
        &\ \norm*{\mathbf{K}_r (I - \mathbf{G} \mathbf{K}_r)^{-1}}_{\mathcal{H}_2} \\ 
        \leq & \ \norm*{\trK(I - \mathbf{G}\trK)^{-1}}_{\mathcal{H}_2} + \norm*{(1 - \mathbf{X} \trDelta)^{-1}}_{\mathcal{H}_\infty}\norm*{(I - \trG \trK)^{-1}}_{\mathcal{H}_\infty} \left(\norm*{\trDelta}_{\mathcal{L}_\infty} \norm*{\trK \mathbf{X}}_{\mathcal{H}_2}  +  \norm*{\trDelta}_{\mathcal{L}_2}\right) 
    \end{align*}
    where the inequality uses \cref{eq:G1_G_2_norm_ineq_H_L_mixed_2_G1_finite_Hinf} in \cref{appendix:background_norm_bounds}.
    From \cref{eq:Kr(I-GKr)_inv_expression}, we know that 
    \begin{align*}
        \mathbf{K}_r (I - \mathbf{G} \mathbf{K}_r)^{-1} \trG = \trK(I - \mathbf{G}\trK)^{-1}\trG + \trK \trX  (I - \trDelta \mathbf{X})^{-1} \mathbf{\Delta}  (I - \mathbf{G} \mathbf{K})^{-1}\trG + (I - \trDelta\mathbf{X} )^{-1}\trDelta (I - \mathbf{G}\trK)^{-1}\trG.
    \end{align*}
    
    
By a similar analysis to the preceding calculation for $\norm*{\trK_r (I - \trG \trK_r)^{-1}}_{\mathcal{H}_2}$, we have
    \begin{align*}
        &\norm*{\mathbf{K}_r (I - \mathbf{G} \mathbf{K}_r)^{-1}\trG}_{\mathcal{H}_2} \\ 
        \leq &\norm*{\trK(I - \mathbf{G}\trK)^{-1}}_{\mathcal{H}_2} + \norm*{(1 - \mathbf{X} \trDelta)^{-1}}_{\mathcal{H}_\infty}\norm*{(I - \trG \trK)^{-1} \trG}_{\mathcal{H}_2} \left(\norm*{\trDelta}_{\mathcal{L}_\infty} \norm*{\trK \mathbf{X}}_{\mathcal{H}_\infty}  +  \norm*{\trDelta}_{\mathcal{L}_\infty}\right) ,
    \end{align*}
    where the inequality uses \cref{eq:G1_G_2_norm_ineq_H_L_mixed_2_G1_finite_Hinf} in \cref{appendix:background_norm_bounds}.
Combining the results above, we have that 
\begin{align*}
J(\mathbf{K}_r) &= \norm*{(I - \mathbf{G}\mathbf{K}_r)^{-1}\mathbf{G}}_{\mathcal{H}_2}^2 + \norm*{(I - \mathbf{G}\mathbf{K}_r)^{-1}\mathbf{G} \trK_r}_{\mathcal{H}_2}^2 + \norm*{\trK_r(I - \mathbf{G}\mathbf{K}_r)^{-1}\mathbf{G}}_{\mathcal{H}_2}^2 + \norm*{\trK_r(I - \mathbf{G}\mathbf{K}_r)^{-1}}_{\mathcal{H}_2}^2 \\
&\leq  \norm*{(1 - \trX \trDelta)^{-1}}_{\mathcal{H}_\infty}^2 (J(\mathbf{K}) + S_1 + S_2).
\end{align*}
where (recalling the notation that $\mathbf{Y} = (I - \mathbf{G}\trK)^{-1}, \mathbf{X} = (I - \trG \trK)^{-1}$),
{\small 
\begin{align*}
    &S_1 = 2\!\left(\norm*{\mathbf{\Delta}}_{L_\infty} \norm*{\mathbf{X}}_{\mathcal{H}_2}\norm*{\mathbf{X} \mathbf{K}}_{\mathcal{H}_2}\!+\!\norm*{\trK \mathbf{Y}}_{\mathcal{H}_2}\! \left(\norm*{\mathbf{\Delta}}_{L_2}\norm*{\mathbf{Y}}_{\mathcal{H}_\infty}\!\!+\!\norm*{\mathbf{\Delta}}_{L_\infty}\!\left(\norm*{\mathbf{Y}}_{\mathcal{H}_2}\! +\!\norm*{\trK \mathbf{X}}_{\mathcal{H}_2}\norm*{\mathbf{Y}}_{\mathcal{H}_\infty}\! \!+\! \norm*{\trK \mathbf{X}}_{\mathcal{H}_\infty}\norm*{\mathbf{X}}_{\mathcal{H}_2} \!\right)\!\right)\!\right)\!, \\
&S_2 = \norm*{\mathbf{\Delta}}_{L_\infty}^2 \norm*{\mathbf{X}}_{\mathcal{H}_2}^2 
+ \left(\norm*{\mathbf{Y}}_{\mathcal{H}_\infty} \left(\norm*{\trDelta}_{\mathcal{L}_\infty} \norm*{\trK \mathbf{X}}_{\mathcal{H}_2}  +  \norm*{\trDelta}_{\mathcal{L}_2}\right) \right)^2
    +  \norm*{\mathbf{X}}_{\mathcal{H}_2}^2\left(\norm*{\trDelta}_{\mathcal{L}_\infty} \norm*{\trK \mathbf{X}}_{\mathcal{H}_\infty}  +  \norm*{\trDelta}_{\mathcal{L}_\infty}\right)^2.
\end{align*}
}

\section{More numerical examples} 
\label{appendix:simulation}

\begin{figure}[t]
\centering
\begin{subfigure}{.5\textwidth}
  \centering
  \includegraphics[width=.8\linewidth]{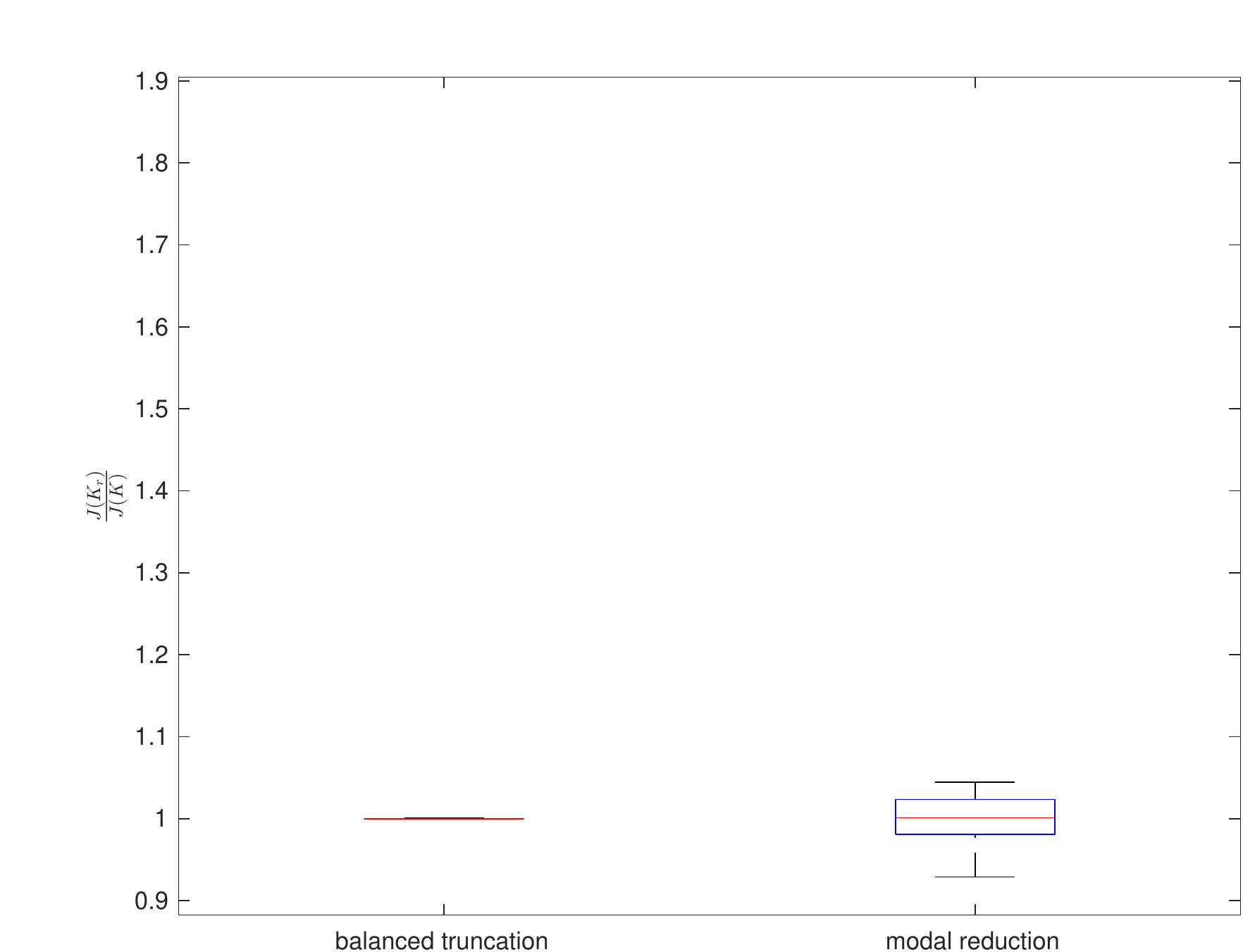}
  \caption{Box plot of LQG cost ratio $\frac{J(K_r)}{J(K)}$}
  \label{fig:bt_mt_comparison_boxplot_cost}
\end{subfigure}%
\begin{subfigure}{.5\textwidth}
  \centering
  \includegraphics[width=.8\linewidth]{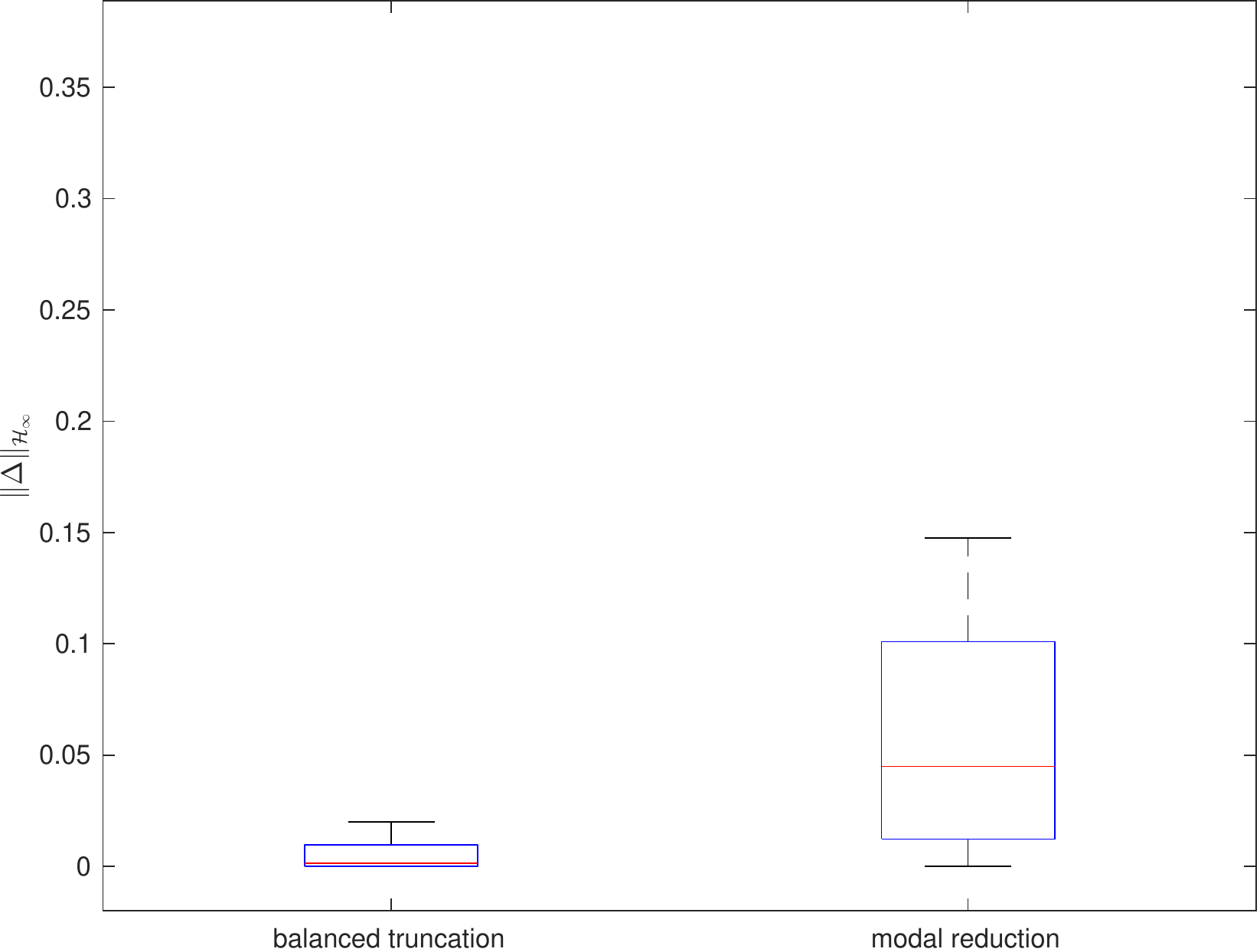}
  \caption{Boxplot of $\norm*{\Delta}_{\mathcal{H}_\infty}$}
  \label{fig:bt_mt_comparison_boxplot_delta}
\end{subfigure}
\caption{Boxplot of LQG cost ratio $\frac{J(K_r)}{J(K)}$ and size of truncated component $\norm*{\Delta}_{\mathcal{H}_\infty}$ for balanced truncation and modal truncation approaches, across 30 trials. The lower and upper blue lines indicate the 25th percentile (Q1) and 75th percentile (Q3) respectively. The lower and upper (shorter) black bars indicate the minimum and maximum non-outlier values, where an outlier is any value that lies below Q1 - 1.5 IQR or Q3 + 1.5 IQR, where IQR refers to the interquartile range, i.e. Q3 - Q1. Note that the relative to the boxplots for modal truncation, the boxplots for balanced truncation are significantly narrower and concentrates about 1.}
    \label{f}
\label{fig:test}
\end{figure}



\textbf{Comparing balanced truncation to modal reduction.} We provide here more comparison between the performance of balanced truncation and that of modal truncation. To do so, we generate a random stable and minimal SISO system of order $3$ which we denote as $\trK_<$, and augment it with an unstable mode $\trK_{\geq}$, which has the form
$$\trK_{\geq} = \left[\begin{array}{c|c}
    0.2 & 0.5 \\
    \hline
    0.5 & 0
\end{array} \right].$$
We define the augmented SISO system $\trK$ as by setting $\trK = \trK_< + \trK_\geq$; note its order is 4. We then compute a stabilizing controller for $\trK$, which we call $\trG$. Treating $\trG$ as the system plant (using the plant-controller duality in SISO systems), we then perform balanced truncation and modal reduction (each reducing the system order by 1) on the stable part of $\trK$ as illustrated in Algorithm \ref{algorithm:balanced_truncation_stable_unstable} and Algorithm \ref{algorithm:modal_truncation}, yielding the reduced order controllers (of order 3) $\trK_{r,\mathrm{bt}}$ and $\trK_{r,\mathrm{mt}}$ respectively. As the box plot in \cref{fig:bt_mt_comparison_boxplot_cost} illustrates, balanced truncation tends to yield reduced-order controllers that essentially has the same LQG cost as the original controller, while modal truncation has a larger spread. For a reason why this might be true, consider the boxplot in \cref{fig:bt_mt_comparison_boxplot_delta}, which also indicates that $\norm*{\trDelta_{r,\mathrm{mt}}}_{\mathcal{H}_\infty}$ has significantly more variability than $\norm*{\trDelta_{r,\mathrm{bt}}}_{\mathcal{H}_\infty}$. This suggests a positive relationship between the size of the truncated component  $\norm*{\trDelta}_{\mathcal{H}_\infty}$, and the deviation of the truncated controller cost from the original controller's cost (in either direction). For the case where $J(\trK_r) > J(\trK)$, this is consistent with our result in \cref{theorem:stable_K_truncation_result}, which states that upper bound on $J(\trK_r) - J(\trK)$ is tighter when $\norm*{\trK - \trK_r}_{\mathcal{H}_\infty}$ is smaller. 

For the practitioner, this suggests that it may be worthwhile to consider both modal and balanced truncation approaches, since for some systems modal truncation may yield relatively large improvements in LQG cost, while for some balanced truncation may be a better choice.

\bibliographystyle{IEEEtran}
\bibliography{biblio}

\end{document}